\documentclass{daj}
\usepackage{amsfonts,amsmath,amssymb,amsthm}
\usepackage{commath,mathtools}

\newcommand{\R}{\mathbb{R}}

\newcommand{\cB}{\mathcal{B}}

\newcommand{\DD}{\mathcal{D}}

\newcommand{\cM}{\mathcal{M}}
\newcommand{\cQ}{\mathcal{Q}}
\newcommand{\cP}{\mathcal{P}}

\newcommand{\cG}{\mathcal{G}}
\newcommand{\cL}{\mathcal{L}}

\newcommand{\cK}{\mathcal{K}}
\newcommand{\HH}{\mathcal{H}}
\newcommand{\TT}{\mathbb{S}^{d-1}}
\newcommand{\T}{\mathcal{T}}

\DeclareMathOperator{\supp}{supp}

\DeclareMathOperator{\diam}{diam}

\DeclareMathOperator{\vis}{Vis}
\DeclareMathOperator{\dimh}{\dim_{\mathrm{H}}}

\newtheorem{theorem}{Theorem}
\newtheorem*{theorem*}{Theorem}
\newtheorem*{conj*}{Conjecture}
\newtheorem{lemma}[theorem]{Lemma}

\newtheorem{prop}[theorem]{Proposition}
\theoremstyle{remark}

\newtheorem*{definition*}{Definition}
\newtheorem{question}[theorem]{Question}
\newtheorem{remark}[theorem]{Remark}

\numberwithin{equation}{section}
\numberwithin{theorem}{section}
\numberwithin{figure}{section}

\dajAUTHORdetails{%
  title = {Visible Parts and Slices of Ahlfors Regular Sets}, %% please capitalize all significant words
  author = {Damian D\k{a}browski},
  plaintextauthor = {Damian Dabrowski},
  keywords = {visibility conjecture, Hausdorff dimension, Marstrand's slicing theorem},
}   %%% END \dajAUTHORdetails

\dajEDITORdetails{%
   year={2024},
   number={17},
   received={27 June 2023},   % received date: example: 7 January 2017   
   revised={24 January 2024},    % Optional revised date (you may comment it out)   
   published={20 December 2024},  % published date
   doi={10.19086/da.127570},       % XXX = number of paper, e.g. da006 for paper#6
}   %%% END \dajEDITORdetails

\begin{document}

\begin{frontmatter}[classification=text]
%% EDITOR: this will force the keywords to appear right after the Abstract.
%%   If the abstract is too long and would force the keywords off the
%%   front page, please comment out % [classification=text] above
%%   This way the keywords will be floated on the bottom of the first page
%%   even though the Abstract spills over to the next page.

%%% AUTHOR: Title goes here.  This line is optional.  You must use it
%%   if title has footnote attached or requires nontrivial typesetting,
%%   e.g., inclusion of linebreaks to force nice layout.
%\title{Visible Parts and Slices\\ of Ahlfors Regular Sets} %% please capitalize all significant words

%%% AUTHOR:
%%% List all authors. If you wish, place grant acknowledgements in \thanks.
%%% In brackets include a short tag for each author.
\author[damian]{Damian D\k{a}browski\thanks{Supported by the Research Council of Finland postdoctoral grant \textit{Quantitative rectifiability and harmonic measure beyond the Ahlfors-David-regular setting}, grant No. 347123.}}

%%% AUTHOR: Abstract goes here
\begin{abstract}
We show that for any compact set $E\subset\R^d$ the visible part of $E$ has Hausdorff dimension at most $d-1/6$ for almost every direction. This improves recent estimates of Orponen and Matheus.

If $E$ is $s$-Ahlfors regular, where $s>d-1$, we prove a much better estimate. In that case we have for almost every direction $\theta$
$$\dimh(\vis_\theta(E)) \le s - \alpha(s-d+1), $$
where $\alpha>0.183$ is absolute. The estimate is new even for self-similar sets satisfying the open set condition.	
Along the way, we prove a refinement of the Marstrand's slicing theorem for Ahlfors regular sets.
\end{abstract}
\end{frontmatter}

%%% AUTHOR: body of paper starts here
 \section{Introduction}
 \subsection{Visibility conjecture}
 Given $\theta\in\mathbb{S}^{d-1}$ let $\ell_\theta=\{t\theta\ :\ t\ge 0\}\subset\R^d$ be the closed half-line spanned by $\theta$. The \emph{visible part} of a compact set $E\subset\R^d$ in direction $\theta$ is defined as
 \begin{equation*}
 \vis_\theta(E) = \{x\in E\ :\ (x+\ell_\theta)\cap E = \{x\} \}.
 \end{equation*}
 Let $\pi_\theta:\R^d\to\theta^\perp$ be the orthogonal projection to the $(d-1)$-dimensional plane $\theta^\perp$. Since $\pi_\theta(\vis_\theta(E))=\pi_\theta(E)$, it follows from the Marstrand-Mattila projection theorem (see e.g. \cite[Theorem 5.8]{mattila2015fourier}) that for $\HH^{d-1}$-a.e. $\theta\in\mathbb{S}^{d-1}$
 \begin{equation}\label{eq:viscon}
 \dimh(\vis_\theta(E))\ge \min\{\dimh(E),\,d-1\},
 \end{equation}
 where $\dimh$ stands for the Hausdorff dimension. Observe that we also have the trivial upper bound
 \begin{equation}\label{eq:trivial}
 \dimh(\vis_\theta(E))\le \dimh(E)
 \end{equation}
 simply because $\vis_\theta(E)\subset E$. The two estimates together imply that for sets satisfying $\dimh(E)\le d-1$ the inequality \eqref{eq:viscon} is in fact an equality. The \emph{visibility conjecture} asserts that the same holds for all compact sets $E$.
 
 \begin{conj*}
 	If $E\subset\R^d$ is compact and $\dimh(E)>d-1$, then for $\HH^{d-1}$-a.e. $\theta\in\mathbb{S}^{d-1}$
 	\begin{equation}\label{eq:conj}
 	\dimh(\vis_\theta(E))= d-1.
 	\end{equation}
 \end{conj*}
 To the best of our knowledge the conjecture first appeared in \cite{jarvenpaa2003visible}; another early mention is \cite[Problem 11]{mattila2004hausdorff}.
 
 The visibility conjecture has been confirmed for several special classes of sets. It holds for quasicircles, graphs of continuous functions \cite{jarvenpaa2003visible}, fractal percolation \cite{arhosalo2012visible}, and for self-similar and self-affine sets satisfying various additional hypotheses \cite{jarvenpaa2003visible,falconer2013visible,rossi2021visible,jarvenpaa2022dimensions}.
 
 Some progress has also been made towards proving \eqref{eq:conj} for general compact sets. A special case of Theorem 1.1 from \cite{jarvenpaa2004transversal} gives that if $E\subset\R^d$ satisfies $0<\HH^s(E)<\infty$, then $\HH^s(\vis_\theta(E))=0$ for a.e. $\theta\in\mathbb{S}^{d-1}$. In a recent breakthrough, Orponen showed that for any compact $E\subset\R^d$ we have for a.e. $\theta\in\mathbb{S}^{d-1}$ 
 \begin{equation}\label{eq:orpo}
 \dimh(\vis_\theta(E))\le d-\frac{1}{50d},
 \end{equation}
 so that at least for sets with dimension close enough to $d$ we can beat the trivial bound \eqref{eq:trivial}. Orponen's proof was optimised by Matheus \cite{matheus2021some}, but due to typos in equations \cite[(2.9) and (3.17)]{matheus2021some} the estimates obtained are weaker than stated.
 
 Quite amazingly, in general it is still not known whether for sets satisfying $\dimh(E)>d-1$ we have $\dimh(\vis_\theta(E))< \dimh(E)$ for a.e. $\theta\in\mathbb{S}^{d-1}$. Until now, this was open even for self-similar sets satisfying the open set condition. If we assume additionally that the rotation group of the associated IFS is finite, and that the projection $\pi_\theta(E)$ has non-empty interior, then this was shown in \cite[Theorem 2.11]{jarvenpaa2022dimensions}. Some estimates for self-similar sets were also shown in \cite{matheus2021some}, but they were only valid if $\dimh(E)$ was large enough, and due to the typo in \cite[(3.17)]{matheus2021some} it is not clear to us what  the estimate obtained for $\dimh(\vis_\theta(E))$ is. 
 
 \vspace{1em}
 In the first result of this article we improve on the trivial bound \eqref{eq:trivial} for all Ahlfors regular sets with $\dimh(E)>d-1$. Recall that a compact set $E\subset\R^d$ is $s$-Ahlfors regular, where $0<s\le d$, if there exists a constant $C\ge 1$ such that
 \begin{equation*}
 C^{-1}r^s\le \HH^s(E\cap B(x,r))\le Cr^s\quad\text{for all $x\in E,\, r>0$.}
 \end{equation*}
 
 \begin{theorem}\label{thm:main}
 	Let $d-1<s\le d$. Suppose that $E\subset \R^d$ is a compact $s$-Ahlfors regular set. Then, for a.e. $\theta\in\TT$
 	\begin{equation}\label{eq:dimest}
 	\dimh (\vis_\theta(E)) \le s - \alpha(s-d+1) = (1-\alpha)s + \alpha(d-1),
 	\end{equation}
 	where $\alpha = 1-\sqrt{6}/3 >0.1835.$
 \end{theorem}
 In particular, the result holds for self-similar sets satisfying weak separation condition, see \cite[Theorem 2.1]{fraser2015assouad}. Note that proving \eqref{eq:dimest} with $\alpha=1$ would establish the visibility conjecture for Ahlfors-regular sets.
 
 In \cite[Remark 1.3]{orponen2022visible} Orponen remarked that while the constant $1/50$ in \eqref{eq:orpo} could be improved by optimising the argument (this was done in \cite{matheus2021some}), ``it seems likely that more ideas will be needed to get an upper bound of the form $n-c$ for some absolute $c>0$.'' In our second main result we obtain such a bound.
 \begin{theorem}\label{thm:main2}
 	Suppose that $E\subset \R^d$ is compact. Then, for a.e. $\theta\in\TT$
 	\begin{equation}\label{eq:dimest2}
 	\dimh (\vis_\theta(E)) \le d-\frac{1}{6}.
 	\end{equation}
 \end{theorem}
 This improves the best known upper bound for $\dimh (\vis_\theta(E))$ for general compact sets.
 
 \subsection{Slicing Ahlfors regular sets}
 When proving \thmref{thm:main} we needed an upper bound for the dimension of \emph{line slices} of Ahlfors regular sets, that is, sets of the form $F\cap \ell$, where $\ell$ is a line. We prove a more general result, valid for slices of arbitrary dimension.
 
 Let $1\le n\le d-1$ be an integer. The celebrated Marstrand's slicing theorem \cite{marstrand1954some,mattila1975hausdorff} states that for a set $F\subset\R^d$ with $0<\HH^s(F)<\infty$ and $n<s\le d$ we have for $\gamma_{d,n}$-a.e. $L\in\mathbb\cG(d,n)$
 \begin{equation}\label{eq:marstrand}
 \dimh(F\cap V_{x,L})=s-n\quad \text{for $\HH^s$-a.e. $x\in F$}.
 \end{equation}
 Here $\cG(d,n)$ denotes the Grassmannian of $n$-dimensional planes in $\R^d$, $\gamma_{d,n}$ is the Haar measure on $\cG(d,n)$, and $V_{x,L}\coloneqq x+L^{\perp}$. See \cite[Theorem 10.11]{mattila1999geometry} or \cite[Theorem 6.9]{mattila2015fourier} for textbook references.
 
 Refinements and alternative proofs of Marstrand's slicing theorem have been found in \cite{mattila1981integralgeometric,orponen2014slicing,mattila2016hausdorff}, see also \cite[Chapter 6]{mattila2015fourier} and recent surveys \cite{mattila2021haudorffdim,mattila2023surveyinter}. The properties of slices of certain special classes of fractals have also been extensively studied, see e.g. \cite{wu2019furstenberg,shmerkin2019furstenberg,algom2020slicing,barany2021finer,anttila2023slices}.
 
 In \cite[p. 97]{mattila2015fourier} Mattila asked whether it is possible to get a dimension upper bound on the set
 \begin{equation*}
 F_{u,L}\coloneqq\{x\in F\ : \ \dimh(F\cap V_{x,L})>u\},
 \end{equation*}
 where $0<\HH^s(F)<\infty$, $n<s\le d$, and $s-n\le u< d-n$. Note that \eqref{eq:marstrand} gives only $\HH^s(F_{u,L})=0$ for a.e. $L\in\cG(d,n)$. We prove such estimate for the endpoint $u=s-n$ in the case that $F$ is Ahlfors regular. In this case, we are also able with no extra effort to replace in the definition of $F_{u,L}$ the Hausdorff dimension by the upper box-counting dimension.
 
 %	Given $\theta\in\mathbb{S}^{d-1}$ let $\cL_\theta$ denote the family of affine lines parallel to $\theta$. Given a compact set $F\subset\R^d$ with $\dimh(F)=s>d-1$, we set
 %	\begin{equation*}
 %	\cL_{\theta,H}=\{\ell\in\cL_\theta\ :\ \overline{\dim}_{\mathrm{box}}(F\cap \ell)> s-d+1\}.
 %	\end{equation*}
 %	and
 %	\begin{equation*}
 %	F_{\theta,H} = F\cap \bigcup_{\ell\in\cL_{\theta,H}}\ell.	
 %	\end{equation*}
 
 \begin{theorem}\label{thm:main3}
 	Let $n<s\le d$. Suppose that $F\subset \R^d$ is a compact $s$-Ahlfors regular set. Then, for $\gamma_{d,n}$-a.e. $L\in \cG(d,n)$
 	\begin{equation}\label{eq:slicing}
 	\dimh(\{x\in F\ : \ \overline{\dim}_{\mathrm{box}}(F\cap V_{x,L})> s-n\})\le n.
 	\end{equation}
 \end{theorem}
 For a more quantitative version of \thmref{thm:main3}, see \propref{prop:heavypart}.
 
 \begin{remark}
 	In fact, we prove that $\eqref{eq:slicing}$ holds for {every} $L\in\cG(d,n)$ such that the push-forward of $\HH^s|_F$ under the orthogonal projection $\pi_L$ belongs to the Sobolev space $H^\sigma(\R^n)$ for all $0<\sigma<(s-n)/2$. This is known to be true for $\gamma_{d,n}$-a.e. $L\in\cG(d,n)$ (see \lemref{lem:sobolevest}), and for a general set $F\subset\R^d$ with $0<\HH^s(F)<\infty$ this is likely to be sharp, in the sense that the set of exceptional planes $\mathcal{E}\subset\cG(d,n)$ where the Sobolev bound fails may satisfy $\dimh(\mathcal{E})=\dimh(\cG(d,n))$. However, it seems plausible that for Ahlfors regular sets $F$ the set of exceptional planes $\mathcal{E}$ is much smaller. Hence, proving any non-trivial upper bounds on $\dimh(\mathcal{E})$ for Ahlfors regular sets would immediately give an estimate on the dimension of exceptional planes where $\eqref{eq:slicing}$ fails.
 \end{remark}
 \begin{remark}
 	The estimate \eqref{eq:slicing} is only interesting if $s\le 2n$, because for $s>2n$ it is easy to show that for $\gamma_{d,n}$-a.e. $L\in \cG(d,n)$
 	\begin{equation}\label{eq:slicing2}
 	\{x\in F\ : \ \overline{\dim}_{\mathrm{box}}(F\cap V_{x,L})> s-n\}=\varnothing,
 	\end{equation}
 	see the argument above \eqref{eq:17}. \eqref{eq:slicing2} also holds trivially for $s=d$, simply because $\dimh(V_{x,L})=d-n$.
 \end{remark}
 
 \begin{question}
 	It is not clear to us how sharp \eqref{eq:slicing} is. Comparing \eqref{eq:slicing} and \eqref{eq:slicing2}, it is tempting to ask whether for $n<s\le \min(2n,d)$ and sets with $0<\HH^s(F)<\infty$ we have
 	\begin{equation*}
 	\dimh(\{x\in F\ : \ \dimh(F\cap V_{x,L})> s-n\})\le 2n-s
 	\end{equation*}
 	for a.e. $L\in\cG(d,n)$. 
 	
 	Observe that if this estimate was true, then for $s>3n/2$ it would automatically improve to $\{x\in F\ : \ \dimh(F\cap V_{x,L})> s-n\}=\varnothing$ for a.e. $L\in\cG(d,n)$, simply because $s-n>2n-s$ in this regime.
 \end{question}
 %	and more generally
 %	\begin{equation*}
 %	\dimh(\{x\in F\ : \ \dimh(F\cap V_{x,L})> u\})\le \max(0,n-u)
 %	\end{equation*}
 %	for $s-n\le u\le d-n$.
 %	

 %	We remark that it follows from the Furstenberg set estimate in \cite[Theorem 1.6]{fu2021incidence} that if $F\subset\R^2$ satisfies $\dimh(F)=s>1$, and $s-1\le u\le 1$, then the family of affine lines $\cL_u = \{\ell\in \mathcal{A}(2,1)\ :\ \dimh(\ell\cap F)>u\}$ satisfies $\dimh(\cL_u)\le 2-u$ (here $\mathcal{A}(2,1)$ denotes the affine Grassmannian).
 
 \subsection{About the proofs}
 The proofs of \thmref{thm:main} and \thmref{thm:main2} follow the general strategy invented in \cite{orponen2022visible}, which we briefly describe below. 
 
 Given a compact set $E\subset [0,1]^d$ and a fixed direction $\theta\in\mathbb{S}^{d-1}$, we seek an estimate for $\HH^{d-\tau}_\infty(\vis_\theta(E))$ for some $\tau>0$. Consider the lines $\cL_\theta$ parallel to $\theta$ that cover $[0,1]^d$. The lines are divided into two classes, the good lines and the bad lines. Loosely speaking, a line $\ell\in\cL_\theta$ is good if the slice $E\cap \ell$ is quite similar to the $\delta$-slice $E\cap \ell(\delta)$, where $\ell(\delta)$ is the $\delta$-neighbourhood of $\ell$. Otherwise, $\ell$ is bad. Let $L_B$ be the union of bad lines, and $L_G$ the union of good lines. We estimate $\HH^{d-\tau}_\infty(\vis_\theta(E)\cap L_B)$ and $\HH^{d-\tau}_\infty(\vis_\theta(E)\cap L_G)$ separately. 
 
 We have a poor understanding of how $\vis_\theta(E)\cap L_B$ looks, but luckily we can show that there are very few bad lines, in the sense that $\pi_\theta(L_B)$ has small dimension. This step uses crucially the Sobolev estimate for the projections of high-dimensional Frostman measures, see \lemref{lem:sobolevest}. Then, we can estimate crudely
 \begin{equation}\label{eq:bad}
 \HH^{d-\tau}_\infty(\vis_\theta(E)\cap L_B)\le \HH^{d-\tau-1}_\infty(\pi_\theta(L_B))\lesssim \delta^\varepsilon.
 \end{equation}
 On the other hand, the good part $\vis_\theta(E)\cap L_G$ has a very nice structure, so that for each $\delta$-tube $T$ parallel to $\theta$ we have
 \begin{equation}\label{eq:good}
 N(T\cap\vis_\theta(E)\cap L_G,\delta)\lesssim\delta^{\tau-1+\varepsilon},
 \end{equation}
 where $N(\cdot,\delta)$ stands for $\delta$-covering number, and then
 \begin{equation*}
 \HH^{d-\tau}_\infty(\vis_\theta(E)\cap L_G)\lesssim \sum_{T} N(T\cap\vis_\theta(E)\cap L_G,\delta)\cdot \delta^{d-\tau}\lesssim \delta^{1-d}\cdot \delta^{\tau-1+\varepsilon}\cdot \delta^{d-\tau}=\delta^\varepsilon.
 \end{equation*}
 
 This is the outline of the proof in \cite{orponen2022visible}. In this paper we make two improvements. Firstly, in \cite{orponen2022visible} one begins by discarding exceptional directions where the Sobolev norm of projections is large (this is important for the estimate of $\HH^{d-\tau-1}_\infty(\pi_\theta(L_B))$). In the end, this forces $\tau\lesssim 1/d$. In the current paper, we do not discard exceptional directions. Instead, all our estimates depend on the Sobolev norm of the projections, and only at the end we average over all the directions. This removes some annoying terms from the estimates, and allows us to obtain \thmref{thm:main2}, where $\tau$ is independent of $d$. 
 
 The second improvement consists of using the slicing estimate from \thmref{thm:main3} (or rather \propref{prop:heavypart}) in the case $n=d-1$, which corresponds to slicing with lines. In \cite{orponen2022visible} it is hard to do better than \eqref{eq:bad} or \eqref{eq:good} because a priori each slice might be of dimension 1. Using \thmref{thm:main3}, we know that for an $s$-Ahlfors regular set $E$ most slices will have dimension at most $s-d+1$, which allows us to significantly improve \eqref{eq:bad} and \eqref{eq:good}. This is how we get \thmref{thm:main}. 
 
 Finally, the proof of \thmref{thm:main3} relies on the Sobolev regularity of projections (\lemref{lem:sobolevest}) and the trace formulas for Sobolev functions (\thmref{thm:traces}). It involves proving a weak-type estimate for the Hardy-Littlewood maximal function which is roughly of the form $\cM:H^{\sigma}(\R^n)\to L^{2,\infty}(\HH^{n-2\sigma}_\infty)$ for $0<\sigma<n/2$, see \lemref{lem:heavyaux}.
 
 \section{Preliminaries}
 \subsection{Notation}
 Inequalities of the form $f\le Cg$, where $C$ is some constant, will be abbreviated as $f\lesssim g$. If the constant $C$ depends on an additional parameter $\varepsilon$, we will write $f\lesssim_{\varepsilon} g$. We will ignore dependence on dimension $d$ and on Ahlfors-regularity constants. If we have $g\lesssim f\lesssim g$, we will write $f\sim g$.
 
 Given $x\in\R^d, r>0$ we will denote the open ball centered at $x$ of radius $r$ by $B(x,r)$. If $x\in\R^{d-1}$, then $B(x,r)$ will stand for the $(d-1)$-dimensional open ball in $\R^{d-1}$. If $B$ is a ball, then $r(B)$ will denote its radius.
 
 If $A\subset\R^d$ and $\delta>0$, then $N(A,\delta)$ denotes the $\delta$-covering number of $A$. 
 
 Given an affine line $\ell$ and $\delta>0$, $\ell(\delta)$ will denote the $\delta$-neighbourhood of $\ell$.
 
 Let $\mathbb{D}$ be the family of standard dyadic cubes on $\R^d$. Given a dyadic cube $Q$ we will denote its sidelength by $\ell(Q)$.
 
 $\cG(d,n)$ stands for the Grassmannian of $n$-dimensional planes in $\R^d$. Given $L\in\cG(d,n)$, $\pi_L:\R^d\to L\simeq \R^n$ denotes the orthogonal projection to $L$. Given $\theta\in\mathbb{S}^{d-1}$, we set $\pi_\theta\coloneqq\pi_{\theta^\perp}$, so that $\pi_\theta:\R^d\to\R^{d-1}$.
 \subsection{Sobolev spaces}	
 For $\sigma\in\R$ the fractional Sobolev space $H^\sigma(\R^d)\subset \mathcal{S}'(\R^d)$ is the completion of the set of Schwartz functions $f\in\mathcal{S}(\R^d)$ such that
 \begin{equation*}
 \|f\|_{H^\sigma(\R^d)}^2\coloneqq \int |\widehat{f}(\xi)|^2(1+|\xi|^2)^{\sigma}\ d\xi<\infty.
 \end{equation*}
 The quantity above is the inhomogeneous Sobolev norm. Its homogeneous counterpart is
 \begin{equation*}
 \|f\|_{\dot{H}^\sigma(\R^d)}^2\coloneqq \int |\widehat{f}(\xi)|^2|\xi|^{2\sigma}\ d\xi.
 \end{equation*}
 
 Observe that if $\nu$ is a finite Borel measure on $\R^d$, and $\sigma\ge 0$, then we can estimate the inhomogeneous norm using the homogeneous one:
 \begin{multline}\label{eq:7}
 \|\nu\|_{H^\sigma(\R^d)}^2 =  \int |\widehat{\nu}(\xi)|^2(1+|\xi|^2)^{\sigma}\ d\xi\\ \lesssim \int_{B(0,1)} |\widehat{\nu}(\xi)|^2\ d\xi +  \int_{B(0,1)^c} |\widehat{\nu}(\xi)|^2|\xi|^{2\sigma}\ d\xi\\
 \lesssim \|\widehat{\nu}(\xi)\|_{L^\infty}^2 + \|\nu\|_{\dot{H}^\sigma(\R^d)}^2 \le \nu(\R^d)^2 + \|\nu\|_{\dot{H}^\sigma(\R^d)}^2.
 \end{multline}
 
 %	It is very well-known that $\mathcal{M}$ is bounded on $L^p(\R^d)$ for all $1<p\le \infty$. Moreover, it preserves also certain smoothness properties of functions. This was first observed by Kinnunen for the standard Sobolev spaces $W^{1,p}$ \cite{kinnunen1997hardy}. The case of fractional Sobolev spaces is due to Korry \cite{korry2004classof}.
 %	\begin{theorem}[{\cite{korry2004classof}}]\label{thm:maxim}
 %		Let $0\le \sigma <1$. If $f\in H^{\sigma}(\R^d)$, then
 %		\begin{equation*}
 %		\|\mathcal{M}f\|_{H^{\sigma}(\R^d)}\le C \|f\|_{H^{\sigma}(\R^d)},
 %		\end{equation*}
 %		where $C=C(\sigma, d)$.
 %	\end{theorem}
 Another tool we will use in our proof is the theory of traces of Sobolev functions with respect to fractal measures. Results of this type can be found e.g. in the classical monographs \cite[Chapter 11]{mazya2011sobolev}, \cite[Chapter 7]{adams1996function}, or \cite[Chapter 7]{triebel2006theory}. The statement we found particularly convenient was Theorem 7.16 from \cite{triebel2006theory}, and in our setting it gives the following.
 \begin{theorem}\label{thm:traces}
 	Let $0<t<d$, $\varepsilon>0$, and $\sigma\ge(d-t+\varepsilon)/2$. Suppose that $f\in \mathcal{S}(\mathbb{R}^d)$, and that $\eta$ is a Radon probability measure on $\R^d$ satisfying $\eta(B(x,r))\le C_0\, r^t$. Then, 
 	\begin{equation*}
 	%		\int |f(x)|^2\, d\eta(x) 
 	\|f\|_{L^2(\eta)}\lesssim_\varepsilon (C_0)^{1/2}\, \|f\|_{H^\sigma(\R^d)}.
 	\end{equation*}
 \end{theorem}

 \subsection{Frostman measures}	
 We are going to use two versions of the classical Frostman's lemma. The first one is due to Orponen \cite[Lemma A.1]{orponen2022visible}, and the difference compared to the usual Frostman's lemma is the lower bound on $\nu(\overline{Q})$.
 \begin{lemma}\label{lem:Frostman}
 	If $K\subset[0,1)^d$ is compact, $0<t\le d$, then there exists a Radon measure $\nu$ with $\supp\nu\subset K$, satisfying 
 	\begin{equation*}
 	\nu(B(x,r))\lesssim r^t\quad\text{for $x\in\R^d, r>0$}
 	\end{equation*}
 	and for all dyadic cubes $Q\subset[0,1)^d$
 	\begin{equation*}
 	\nu(\overline{Q})\gtrsim \min(\HH^{t}_\infty(Q\cap K), \HH^{d}(Q)).
 	\end{equation*}
 \end{lemma}
 We remark that the Frostman measure $\nu$ above satisfies $\nu(K)\sim\HH^{t}_\infty(K)$, and in particular if $t>\dimh(K)$, then $\nu\equiv 0$.
 
 The second version of Frostman's lemma we will use concerns unions of dyadic cubes, and it asserts that at small scales the Frostman measure associated to such sets is just normalized Lebesgue measure. This follows from the usual construction of the Frostman measure using dyadic cubes, as in \cite[Theorem 8.8]{mattila1999geometry}, and we omit the proof.
 \begin{lemma}\label{lem:Frostman2}
 	Let $\cK\subset\mathbb{D}$ be a collection of disjoint dyadic cubes with $K\coloneqq\bigcup_{Q\in\cK}Q\subset [0,1]^d$. For any $0<t\le d$ there exists a Radon measure $\nu$ with $\supp\nu\subset \overline{K}$, satisfying $\nu(K)\sim\HH^t_\infty(K)$,
 	\begin{equation*}
 	\nu(B(x,r))\lesssim r^t\quad\text{for $x\in\R^d, r>0$},
 	\end{equation*}
 	and for all $Q\in\cK$
 	\begin{equation}\label{eq:Frostlower2}
 	\nu|_Q = \frac{\nu(Q)}{\ell(Q)^d}\HH^d|_Q.
 	\end{equation}
 \end{lemma}

 %	Given $0<t<d$, the Riesz $t$-energy of a Radon measure $\nu$ on $\R^d$ is defined as
 %	\begin{equation*}
 %	I_t(\nu) \coloneqq \iint \frac{1}{|x-y|^{t}}\ d\nu(x)d\nu(y).
 %	\end{equation*}
 %	A standard computation gives the following.
 %	\begin{lemma}\label{lem:energyest}
 %		Suppose that $\nu$ is a compactly supported Borel measure on $\R^d$, and that it satisfies $\nu(B(x,r))\le C_0\, r^t$ for some $t\in (0,d]$. Then, for all $0<t'<t$
 %		\begin{equation*}
 %		I_{t'}(\nu) \lesssim_{t'} C_0\,\nu(\R^d).
 %		\end{equation*}
 %	\end{lemma}

 Let $\nu$ be a compactly supported Radon measure on $\R^d$. Given $L\in\cG(d,n)$ we will write $\nu_L$ to denote the push-forward of $\nu$ under the orthogonal projection $\pi_L$. Whenever $\nu_L\ll\HH^{n},$ which will always be the case in this paper, the expression ``$\nu_L$'' will denote both the measure and its density with respect to $\HH^{n}$. Similarly, for $\theta\in\mathbb{S}^{d-1}$ we will write $\nu_\theta$ to denote the push-forward of $\nu$ under the orthogonal projection $\pi_\theta:\R^d\to\R^{d-1}.$
 
 The following estimate is well-known (see e.g. \cite[Theorem 5.10]{mattila2015fourier}), so we only sketch the proof.
 \begin{lemma}\label{lem:sobolevest}
 	Suppose that $\nu$ is Radon measure on $\R^d$ with $\supp\nu\subset [0,1]^d$, and that it satisfies $\nu(B(x,r))\le C_0\, r^t$ for some $n<t\le d$. Let $0<\varepsilon<t-n$, and $\sigma = (t-n-\varepsilon)/2$. Then,
 	\begin{equation*}
 	\int_{\cG(d,n)}\|\nu_L\|_{H^\sigma(\R^{n})}^2\, d\gamma_{d,n}(L) \lesssim_\varepsilon C_0\,\nu(\R^d).
 	\end{equation*}
 \end{lemma}
 \begin{proof}
 	Recall that given $0<s<d$, the Riesz $s$-energy of a Radon measure $\mu$ on $\R^d$ is defined as
 	\begin{equation*}
 	I_s(\mu) \coloneqq \iint \frac{1}{|x-y|^{s}}\ d\mu(x)d\mu(y).
 	\end{equation*}
 	If $\mu$ is compactly supported, then we also have the following well-known identity relating Riesz energy with the homogeneous Sobolev norm: for $0<s<d$
 	\begin{equation}\label{eq:energysobolev}
 	I_s(\mu)=C_{d,s}\int_{\R^d} |\widehat{\mu}(\xi)|^2|\xi|^{s-d}\, d\xi,
 	\end{equation}
 	see \cite[Theorem 3.10]{mattila2015fourier} for the proof.
 	
 	A simple computation shows that the Frostman condition $\nu(B(x,r))\le C_0\, r^t$ implies
 	\begin{equation}\label{eq:rieszener}
 	I_{t-\varepsilon}(\nu)\lesssim_\varepsilon C_0\,\nu(\R^d).
 	\end{equation}
 	Using the fact that $\widehat{\nu_L}(\xi)=\widehat{\nu}(\xi)$ for $\xi\in L$, another standard computation (see \cite[Theorem 5.10]{mattila2015fourier} for details) gives
 	\begin{multline*}
 	\int_{\cG(d,n)}\|\nu_L\|_{\dot{H}^\sigma(\R^{n})}^2\, d\theta = \int_{\cG(d,n)}\int_{L}|\widehat{\nu_L}(\xi)|^2|\xi|^{t-\varepsilon-n}\ d\HH^n(\xi)\, d\gamma_{d,n}(L)\\
 	\sim \int_{\R^d} |\widehat{\nu}(\xi)|^2|\xi|^{t-\varepsilon-d}\ d\xi
 	\sim I_{t-\varepsilon}(\nu)\lesssim_\varepsilon C_0\,\nu(\R^d).
 	\end{multline*}
 	Together with \eqref{eq:7}, this gives the desired estimate.
 \end{proof}	
 
 \section{Slicing Ahlfors regular sets}
 In this section we prove \thmref{thm:main3}. We first prove the following more quantitative result, which will be crucial later in the proof of \thmref{thm:main}. 	
 
 Given a locally integrable function $f:\R^d\to \R$, the centered Hardy-Littlewood maximal function of $f$ is defined as
 \begin{equation*}
 \mathcal{M}f(x) = \sup_{r>0}\frac{1}{r^d}\int_{B(x,r)} |f(y)|\, dy.
 \end{equation*}
 Recall that if $\nu$ is a Radon measure on $\R^d$ and $L\in\cG(d,n)$, then $\nu_L$ is the pushforward of $\nu$ under the orthogonal projection $\pi_L:\R^d\to\R^n$.
 \begin{prop}\label{prop:heavypart}
 	Let $n<s\le \min(2n,d)$ and $0<\varepsilon<(s-n)/2$. Let $\nu$ be an $s$-Ahlfors regular measure on $\R^d$ with $F\coloneqq\supp\nu\subset B(0,1)$. Let $M> 3^n\nu(F)$, $L\in \cG(d,n)$, and
 	\begin{equation*}
 	F_M =\bigg\{x\in F\ :\ \mathcal{M}\nu_L(\pi_L(x))\ge M \bigg\}.
 	\end{equation*}
 	%		For brevity of notation, we set $M_F\coloneqq M \nu(F)\diam(F)^{1-d}\sim M \diam(F)^{s-d+1}$. 
 	Then, for $\sigma = (s-n-\varepsilon)/2$ we have
 	\begin{equation*}
 	\HH^{n+2\varepsilon}_\infty(F_M)\lesssim_{\varepsilon} M^{-1}\|\nu_L\|^2_{H^\sigma(\R^{n})}.
 	\end{equation*}
 \end{prop}
 Most of this section is dedicated to proving this result. 	
 Fix $\nu$ and $F$ as above, and let
 \begin{equation*}
 H' = \{x\in \R^n\ :\ \cM\nu_L(x)\ge M \},
 \end{equation*}
 so that $H' = \pi_L(F_M)$. We would like to get a bound of the form
 \begin{equation*}
 \HH^{2n-s+2\varepsilon}_\infty(H')\lesssim_\varepsilon M^{-2}\|\nu_L\|^2_{H^\sigma(\R^{n})}.
 \end{equation*}
 We first establish a dyadic variant of this estimate.
 
 Let $\cB\subset\mathbb{D}_{\R^{n}}$ be the family of maximal dyadic cubes satisfying
 \begin{equation}\label{eq:nubigonQ}
 \frac{\nu_L(Q)}{\ell(Q)^{n}}\ge 3^{-n}M.
 \end{equation}
 Since $3^{-n}M> \nu(F)$ and $\supp\nu_L\subset [0,1)^n$, we get that all cubes from $\cB$ are contained in $[0,1)^n$. Let $H=\bigcup_{Q\in\cB}Q$.	
 
 Note that for any $Q\in\mathbb{D}_{\R^n}$ which is not contained in $H$ we have
 \begin{equation}\label{eq:nusmall}
 \frac{\nu_L(Q)}{\ell(Q)^{n}}< 3^{-n}M.
 \end{equation}
 
 \begin{lemma}\label{lem:heavyaux}
 	We have
 	\begin{equation}\label{eq:3}
 	\HH^{2n-s+2\varepsilon}_\infty(H)\lesssim_\varepsilon M^{-2}\|\nu_L\|^2_{H^\sigma(\R^{n})}.
 	\end{equation}
 \end{lemma}
 \begin{proof}
 	Assume that $\HH^{2n-s+2\varepsilon}_\infty(H)>0$, since otherwise there is nothing to prove. Let $\eta$ be the Radon measure on $H$ given by Frostman's lemma \lemref{lem:Frostman2}, normalized so that $\eta(\R^n)=1$ (this can be done since $\HH^{2n-s+2\varepsilon}_\infty(H)>0$), so that it satisfies
 	\begin{equation*}
 	\eta(B(x,r))\lesssim (\HH^{2n-s+2\varepsilon}_\infty(H))^{-1} r^{2n-s+2\varepsilon},
 	\end{equation*}
 	where the implicit constant depends only on the dimension.
 	
 	We would like to apply \thmref{thm:traces} with respect to $\nu_L$ and $\eta$. However, it is not clear if we can do that, since $\nu_L$ might not be a Schwartz function. To overcome this minor issue, we will mollify $\nu_L$. 
 	
 	Let $\varphi:\R^{n}\to \R$ be a non-negative radial $C^\infty$ function with $\varphi(0)>0,$ $\supp\varphi\subset B(0,1)$ and $\int\varphi =1$. For $0<\gamma<1$ set $\varphi_{\gamma}(x) = \gamma^{-n}\varphi(x/\gamma)$. Let $f_\gamma = \nu_L\ast\varphi_{\gamma}$.
 	
 	By \thmref{thm:traces} applied to $f_\gamma$ and $\eta$, with $t= 2n-s+2\varepsilon$ and $\sigma = (s-n-\varepsilon)/2$, we have
 	\begin{multline}\label{eq:15}
 	\int_{\R^{n}} |f_\gamma|^2\, d\eta 	\lesssim_\varepsilon (\HH^{2n-s+2\varepsilon}_\infty(H))^{-1} \|f_\gamma\|_{H^\sigma(\R^{n})}^2\\
 	\lesssim (\HH^{2n-s+2\varepsilon}_\infty(H))^{-1} \|\nu_L\|_{H^\sigma(\R^{n})}^2,
 	\end{multline}
 	where in the last estimate we used the fact that convolving with $\varphi_{\gamma}$ does not increase the Sobolev norm. 
 	
 	Now we get a lower bound for the left hand side. First, we claim that if $Q \in \cB$ and $0<\gamma\le \ell(Q)/2$, then
 	\begin{equation}\label{eq:14}
 	\int_Q f_\gamma(x)\, dx \gtrsim M\ell(Q)^{n}.
 	\end{equation}
 	Indeed, let $\widetilde{Q}\subset Q$ be a dyadic subcube with $\ell(\widetilde{Q})=\ell(Q)/2$ and $\nu_L(\widetilde{Q})\gtrsim M\ell(Q)^{n}$ (such subcube exists by \eqref{eq:nubigonQ}). It follows from elementary geometry that the set 
 	\begin{equation*}
 	C \coloneqq \{y\in B(0,\gamma)\ :\ y+\widetilde{Q}\subset Q\}
 	\end{equation*}
 	is a truncated one-sided cone centered at 0 with opening angle $\sim 1$. Since $\varphi_{\gamma}$ is radial and $\int_{B(0,\gamma)}\varphi_{\gamma}=1$, we get $\int_C \varphi_{\gamma}(y)\, dy\gtrsim 1$. Thus,
 	\begin{multline*}
 	\int_Q f_\gamma(x)\, dx=\int_Q \nu_L\ast\varphi_{\gamma}(x)\, dx = \int_{B(0,\gamma)} \varphi_{\gamma}(y)\int_Q \nu_L(x-y)\, dx\, dy\\
 	\ge \int_{C} \varphi_{\gamma}(y)\int_Q \nu_L(x-y)\, dx\, dy\ge \int_{C} \varphi_{\gamma}(y)\int_{\widetilde{Q}} \nu_L(z)\, dz\, dy\\
 	= \nu_L(\widetilde{Q})\int_{C} \varphi_{\gamma}(y)\, dy\gtrsim M\ell(Q)^{n},
 	\end{multline*}
 	where in the last estimate we used that $\nu_L(\widetilde{Q})\gtrsim M\ell(Q)^{n}$, by the definition of $\widetilde{Q}$. This gives \eqref{eq:14}.
 	
 	It follows from the Cauchy-Schwarz inequality, the property \eqref{eq:Frostlower2} of the Frostman measure $\eta$, and \eqref{eq:14}, that if $Q \in \cB$ and $0<\gamma\le \ell(Q)/2$, then
 	\begin{equation*}
 	(\eta(Q))^{1/2}\big(\int_Q |f_\gamma|^2\, d\eta\big)^{1/2}\ge \int_Q f_\gamma\, d\eta = \frac{\eta(Q)}{\ell(Q)^{n}} \int_Q f_\gamma(x)\, dx\gtrsim M\eta(Q),
 	\end{equation*}
 	and so
 	\begin{equation}\label{eq:16}
 	\int_Q |f_\gamma|^2\, d\eta\gtrsim M^2\,\eta(Q).
 	\end{equation}
 	
 	Let $\cB_{\gamma}=\{Q\in \cB \ :\ \ell(Q)/2>\gamma\}$, and $H_\gamma = \bigcup_{Q\in\cB_\gamma} Q$. Note that $\{H_{1/k}\}_{k\in\mathbb{N}}$ is an increasing sequence of sets, with $H = \bigcup_k H_{1/k}$. Hence, by the continuity of measure we have
 	\begin{equation*}
 	1=\eta(H)=\lim_{k\to\infty}\eta(H_{1/k}).
 	\end{equation*}
 	Let $k$ be so large that $\eta(H_{1/k})\ge 1/2$. Then, we get from \eqref{eq:16} and \eqref{eq:15} that
 	\begin{equation*}
 	M^2\lesssim M^2\sum_{Q\in\cB_{1/k}}\eta(Q)\lesssim\int_{H_{1/k}}|f_{1/k}|^2\, d\eta\lesssim (\HH^{2n-s+2\varepsilon}_\infty(H))^{-1} \|\nu_L\|_{H^\sigma(\R^{n})}^2.
 	\end{equation*}
 	This gives the desired inequality \eqref{eq:3}.
 \end{proof}
 
 In order to pass from \eqref{eq:3} to an estimate on the size of $H'= \{x\in \R^n\ :\ \cM\nu_L(x)\ge M \},$ we will use the well-known one-third trick. For every $e\in\{0,1\}^n$ consider the translated dyadic grid on $\R^n$
 \begin{equation*}
 \mathbb{D}^e_{\R^n}=\frac{1}{3}e+\mathbb{D}_{\R^n}.
 \end{equation*}
 The proof of the following lemma can be found e.g. in \cite[Section 3]{lerman2003quantifying}.
 \begin{lemma}\label{lem:one-third}
 	For every $x\in\R^n$ and $0<r<1/3$ there exists $e\in\{0,1\}^n$ and $Q\in\mathbb{D}_{\R^n}^e$ such that $B(x,r)\subset Q$ and $\ell(Q)\le 3r$.
 \end{lemma}
 
 Let $\cB^e\subset\mathbb{D}^e_{\R^n}$ be defined analogously as $\cB$, but using the translated grid $\mathbb{D}^e_{\R^n}$. Set $H^e=\bigcup_{Q\in\cB^e}Q$. By \lemref{lem:heavyaux}, for each $e\in\{0,1\}^n$ we have $\HH^{2n-s+2\varepsilon}_\infty(H^e)\lesssim M^{-2}\|\nu_L\|^2_{H^\sigma(\R^{n})}$.
 
 \begin{lemma}\label{lem:tildeH}
 	We have 
 	\begin{equation*}
 	H'\subset\bigcup_{e\in\{0,1\}^n} H^e.
 	\end{equation*}
 \end{lemma}
 \begin{proof}
 	Let $x\in H'$. Then, there exists some ball $B(x,r)$ such that
 	\begin{equation*}
 	\frac{\nu_L(B(x,r))}{r^n}\ge M.
 	\end{equation*}
 	Since $M>3^n\nu(F)$, we have $0<r<1/3$. By \lemref{lem:one-third}, there exists $e\in\{0,1\}^n$ and $Q\in\mathbb{D}_{\R^n}^e$ such that $B(x,r)\subset Q$ and $\ell(Q)\le 3r$. Thus,
 	\begin{equation*}
 	\frac{\nu_L(Q)}{\ell(Q)^n}\ge3^{-n}\frac{\nu_L(B(x,r))}{r^n}\ge 3^{-n}M,
 	\end{equation*}
 	which gives $x\in H^e$.
 \end{proof}
 
 It follows from \lemref{lem:tildeH} and \lemref{lem:heavyaux} that
 \begin{equation*}
 \HH^{2n-s+2\varepsilon}_\infty(H')\le \sum_{e\in\{0,1\}^n}\HH^{2n-s+2\varepsilon}_\infty(H^e)\lesssim_\varepsilon M^{-2}\|\nu_L\|^2_{H^\sigma(\R^{n})}.
 \end{equation*}
 
 %	Let $\widetilde{\cP}\subset\bigcup_{e\in\{0,1\}^n}\cP^e$ be the collection of cubes which are not strictly contained in any other cube from $\bigcup_{e\in\{0,1\}^n}\cP^e$. Since each $\cP^e$ is a covering of $H^e$, it follows from \lemref{lem:tildeH} that $\widetilde{\cP}$ is a covering of $H'$.

 Let $\cQ\subset\mathbb{D}_{\R^n}$ be a family of disjoint dyadic cubes covering $H'$ and such that
 \begin{equation}\label{eq:Qcoveringdef}
 \sum_{Q\in\cQ}\ell(Q)^{2n-s+2\varepsilon} \lesssim \HH^{2n-s+2\varepsilon}_\infty(H')\lesssim_\varepsilon M^{-2}\|\nu_L\|^2_{H^\sigma(\R^{n})}.
 \end{equation}
 Without loss of generality, we may assume that every $Q\in\cQ$ intersects $H'$. 
 
 For every $Q\in\cQ$ let $Q'\in\mathbb{D}_{\R^n}$ be the smallest cube with ${Q}'\supset Q$ and such that ${Q}'\cap (H')^c\neq\varnothing$, that is, ${Q}'$ is not contained in $H'$. Of course, it may happen that $Q'=Q$. Let $\cP\subset\mathbb{D}_{\R^n}$ be the family of maximal dyadic cubes from $\{Q'\}_{Q\in\cQ}$. Since $\cQ$ is a covering of $H'$, the family $\cP$ is also a covering of $H'$. Note that the cubes in $\cP$ are pairwise disjoint.
 
 %	For each $B\in\cB_0$ let $D_B$ be a dyadic cube with $D_B\subset B$ and $\ell(D_B)\sim r(B)$. Let
 %	\begin{equation*}
 %	\mathbb{D}_{\cB} = \{D_B\ : B\in\cB_0 \}.
 %	\end{equation*}
 %	Note that there exists a dimensional constant $C_1\ge 1$ such that 
 %	\begin{equation*}
 %	H_0\subset \bigcup_{Q\in\mathbb{D}_{\cB}} C_1Q.
 %	\end{equation*}
 %		We define $\cQ_\cB$ to be the family of maximal cubes from $\{Q\in\cQ_B \ :\ B\in\cB_0 \}$. In particular, $\cQ_\cB$ is a covering of $H_0$.
 
 %		We divide the cubes in $\cQ$ into two subfamilies. We say that $Q\in\cQ$ is big, denoted by $Q\in\cQ_{\bbb}$, if there exists a cube $D_B\in\mathbb{D}_\cB$ with $D_B\subset Q$. Otherwise, we will say that $Q$ is small, denoted by $\cQ_{\sss}$.
 
 %		For each $C\in\cC_{\sss}$ let $Q(C)\in \cQ_\cB$ be a cube satisfying $C\subset 3Q$; if there is more than one such cube, we choose the one with maximal sidelength.		
 
 %		\begin{equation*}
 %			H_0 = \bigcup_{B\in\cB_0} 5B\subset \bigcup_{Q\in\cQ} Q\subset \bigcup_{P\in\cP} 100P.
 %		\end{equation*}
 \begin{lemma}\label{eq:4}
 	We have
 	\begin{equation}\label{eq:6}
 	\sum_{P\in\cP}\ell(P)^{2n-s+2\varepsilon}\lesssim_\varepsilon M^{-2}\|\nu_L\|^2_{H^\sigma(\R^{n})}.
 	\end{equation}		
 \end{lemma}
 \begin{proof}
 	We claim that
 	\begin{equation*}
 	\sum_{P\in\cP}\ell(P)^{2n-s+2\varepsilon} \lesssim \sum_{Q\in\cQ}\ell(Q)^{2n-s+2\varepsilon}.
 	\end{equation*}
 	The estimate \eqref{eq:6} will then follow immediately from \eqref{eq:Qcoveringdef}.
 	
 	The bound above is trivial for $\cP\cap \cQ$, so it suffices to show
 	\begin{equation}\label{eq:1}
 	\sum_{P\in\cP\setminus\cQ}\ell(P)^{2n-s+2\varepsilon} \lesssim \sum_{Q\in\cQ}\ell(Q)^{2n-s+2\varepsilon}.
 	\end{equation}
 	Fix a cube $P\in \cP\setminus \cQ$. It follows that $P=Q'$ for some $Q\in\cQ$, and $Q\subsetneq Q'$. Since $Q'$ is the smallest ancestor of $Q$ that is not contained in $H'$, there exists a cube $Q''\in\mathbb{D}_{\R^n}$ with $Q\subset Q''\subsetneq Q'$, $\ell(Q'')=\ell(Q')/2$, and $Q''\subset H'$. At the same time, $\cQ$ covers $H'$, and in particular, it covers $Q''$. Thus, denoting by $\cL^n$ the Lebesgue measure on $\R^n$,
 	\begin{equation*}
 	\ell(P)^n=\ell(Q')^n= 2^n\ell(Q'')^n=2^n\cL^n(Q'') = 2^n\sum_{R\in\cQ,\, R\subset Q''} \cL^n(R) \lesssim \sum_{R\in\cQ,\, R\subset P} \ell(R)^n.
 	\end{equation*}
 	Recalling that $0<2n-s+2\varepsilon< n$, we use the elementary inequality $|\sum {a_i}|^p \le \sum |a_i|^p$ for $p\in (0,1)$ to get
 	\begin{equation*}
 	\ell(P)^{2n-s+2\varepsilon} = \left(\sum_{R\in\cQ,\, R\subset P} \ell(R)^{n}\right)^{(2n-s+2\varepsilon)/n}
 	\lesssim \sum_{R\in\cQ,\, R\subset P} \ell(R)^{2n-s+2\varepsilon}.
 	\end{equation*}
 	Summing over $P\in\cP\setminus\cQ$, and recalling that $\cP$ is a family of disjoint cubes, gives \eqref{eq:1}.
 \end{proof}
 
 Given a dyadic cube $Q$ and a constant $C\ge 1$, we will write $CQ$ to denote the cube with the same center as $Q$ and with sidelength $C\ell(Q)$.	
 \begin{lemma}
 	For any $P\in{\cP}$ and $C\ge 1$ we have
 	\begin{equation}\label{eq:5}
 	\nu_L(CP)\lesssim_C M\ell(P)^{n}.
 	\end{equation}
 \end{lemma}
 \begin{proof}
 	Let $P\in{\cP}$. Then, by the definition of $\cP$, the cube $P$ is not contained in $H'$. Let $x\in P\setminus H'$, and let $C'\sim C$ be such that $B(x,C'\ell(P))\supset CP$. Then,
 	\begin{equation*}
 	\frac{\nu_L(CP)}{\ell(P)^{n}}\le (C')^n\frac{\nu_L(B(x,C'\ell(P))}{(C'\ell(P))^{n}}\le (C')^n M,
 	\end{equation*}
 	where the last inequality follows from the fact that $\cM\nu_L(x)\le M$.
 \end{proof}
 
 %		In particular, if $\|\mu_\theta\|^2_{H^\sigma(\R^{n})}\lesssim \delta^{-\varepsilon}$, then for each $P\in\cP$ we have
 %		\begin{equation*}
 %		\ell(P)\lesssim \delta^{\varepsilon/(2d-2-s+2\varepsilon)}.
 %		\end{equation*}	
 We are ready to conclude the proof of \propref{prop:heavypart}. We remark that until this point we never used lower Ahlfors regularity of measure $\nu$, but now it will be crucial.
 \begin{proof}[Proof of \propref{prop:heavypart}]		
 	For every $P\in\cP$ let $\DD_P\subset\mathbb{D}$ be the family of dyadic cubes $Q\subset [0,1]^d$ with $\ell(Q)=\ell(P),\ \pi_L(Q)\cap P\neq\varnothing,$ and $F\cap Q\neq\varnothing$. Since $F_M=F\cap \pi_L^{-1}(H')$, and $H'\subset \bigcup_{P\in\cP} P$, we get that
 	\begin{equation*}
 	F_M\subset \bigcup_{P\in\cP}\bigcup_{Q\in\mathcal{D}_P}Q.
 	\end{equation*}		
 	Note that, since $\nu$ is $s$-Ahlfors regular, we have $\nu(3Q)\sim \ell(Q)^s$ for all $Q\in\DD_P$, the cubes $\{3Q\}_{Q\in\DD_P}$ have bounded intersections, and they satisfy $\pi_L(3Q)\subset CP$ for some dimensional constant $C\ge 1$. It follows that
 	\begin{multline*}
 	\HH^{n+2\varepsilon}_\infty(F_M)\lesssim \sum_{P\in\cP} \sum_{Q\in \mathcal{D}_P}\ell(Q)^{n+2\varepsilon}\sim \sum_{P\in\cP}\ell(P)^{n+2\varepsilon-s} \sum_{Q\in \mathcal{D}_P}\ell(Q)^{s}\\
 	\sim \sum_{P\in\cP}\ell(P)^{n+2\varepsilon-s} \sum_{Q\in \mathcal{D}_P}\nu(3Q)
 	\lesssim \sum_{P\in\cP}\ell(P)^{n+2\varepsilon-s} \nu_L(CP)\\
 	\overset{\eqref{eq:5}}{\lesssim}  M\sum_{P\in\cP}\ell(P)^{2n+2\varepsilon-s}\overset{\eqref{eq:6}}{\lesssim_\varepsilon} M^{-1}\|\nu_\theta\|^2_{H^\sigma(\R^{n})}.
 	\end{multline*}
 \end{proof}
 \thmref{thm:main3} follows easily from \propref{prop:heavypart}.
 \begin{proof}[Proof of \thmref{thm:main3}]
 	Set $\nu=\HH^s|_F$, let $M> 3^n\nu(F)$ be a large integer. Given $L\in \cG(d,n)$ define
 	\begin{equation*}
 	F_M = \{x\in F\ :\ \cM\nu_L(\pi_L(x))\ge M\}.
 	\end{equation*} 
 	%		Then, it follows from \propref{prop:heavypart} that
 	%		\begin{equation*}
 	%		\HH_\infty^{n+2\varepsilon}\bigg(\bigcup_{n\ge N}F_n\bigg)\le \sum_{n\ge N}\HH_\infty^{n+2\varepsilon}(F_n)\lesssim_{\varepsilon} 2^{-N}\|\nu_\theta\|^2_{H^{\sigma}(\R^{n})}.
 	%		\end{equation*}
 	
 	Let $x\in F$ be such that the $(d-n)$-plane $V=V_{x,L}=x+L^\perp$ satisfies $\overline{\dim}_{\mathrm{box}}(V\cap F)>s-n$. We claim that $V\cap F\subset F_M$.		
 	Indeed, since $\overline{\dim}_{\mathrm{box}}(V\cap F)>s-n$, there exists $\beta>0$ and a sequence $\delta_m\to 0$ such that
 	\begin{equation*}
 	N(F\cap V,\delta_m)\ge \delta_m^{-s+n-\beta}.
 	\end{equation*}
 	Since $\nu$ is $s$-Ahlfors regular, it follows that
 	\begin{equation*}
 	\nu(V(2\delta_m))\gtrsim \delta_m^{-s+n-\beta}\cdot \delta_m^s = \delta_m^{n-\beta},
 	\end{equation*}
 	where $V(2\delta_m)$ is the $2\delta_m$-neighbourhood of $V$.
 	Hence, for $x\in V$ we have
 	\begin{equation*}
 	\cM\nu_L(\pi_L(x))\ge \frac{\nu_L(B(\pi_L(x),2\delta_m))}{(2\delta_m)^{n}}\gtrsim \delta_m^{-\beta}.
 	\end{equation*}
 	If $\delta_m$ is chosen small enough, we get $x\in F_M$.
 	
 	The argument above implies that $\{x\in F\ : \ \overline{\dim}_{\mathrm{box}}(F\cap V_{x,L})> s-n\}\subset F_M$. In the case $s>2n$, it follows from \lemref{lem:sobolevest} and the Sobolev embedding that $\nu_L\in L^\infty(\R^n)$ for $\gamma_{d,n}$-a.e. $L\in\cG(d,n)$, so that $\cM\nu_L$ is a bounded function. Hence, taking $M$ large enough, we get
 	\begin{equation}\label{eq:17}
 	\{x\in F\ : \ \overline{\dim}_{\mathrm{box}}(F\cap V_{x,L})> s-n\}\subset F_M=\varnothing.
 	\end{equation}
 	This gives \eqref{eq:slicing2}.

 	On the other hand, if $n<s\le 2n$, then by \propref{prop:heavypart}
 	\begin{equation*}
 	\HH_\infty^{n+2\varepsilon}(\{x\in F\ : \ \overline{\dim}_{\mathrm{box}}(F\cap V_{x,L})> s-n\})\lesssim_{\varepsilon} M^{-1}\|\nu_L\|^2_{H^\sigma(\R^{n})}.
 	\end{equation*}
 	Taking $M\to\infty$ gives for all $L\in\cG(d,n)$ such that $\|\nu_L\|^2_{H^\sigma(\R^{n})}<\infty$ (which is true for a.e. $L\in\cG(d,n)$)
 	\begin{equation*}
 	\HH_\infty^{n+2\varepsilon}(\{x\in F\ : \ \overline{\dim}_{\mathrm{box}}(F\cap V_{x,L})> s-n\})=0.
 	\end{equation*}
 	Letting $\varepsilon\to 0$ finishes the proof.
 \end{proof}
 
 \section{Decomposition of the visible part}
 We begin the proof of \thmref{thm:main}. Let $E\subset \R^d$ be an $s$-Ahlfors regular set, with $s\in (d-1, d]$. We set $\mu = \HH^s|_{E}$. By rescaling and translating, we may assume that $\diam(E)\sim 1\sim\mu(E)$ and $E\subset [0,1]^d$.
 \subsection{Parameters, cubes, and tubes}\label{sec:params}
 Let $\alpha\in (0,1)$ (in the end we will take $\alpha=1-\sqrt{6}/3$), let $\varepsilon>0$ be a small constant, and set
 \begin{equation}\label{eq:deftau}
 \tau \coloneqq \alpha(s-d+1)-5\varepsilon.
 \end{equation}
 Our goal is to prove that
 \begin{equation}\label{eq:goal}
 \HH_\infty^{s-\tau}(\vis_\theta(E))=0\quad\text{for a.e. $\theta\in\TT$}.
 \end{equation}
 Taking $\varepsilon\to 0$, the estimate \eqref{eq:dimest} will follow.
 
 Fix a small dyadic scale $\delta\in 2^{-\mathbb{N}}$. In the proof we will often assume without further mention that $\delta$ is small enough depending on $\varepsilon$. 
 
 We are going to show the following.
 \begin{prop}\label{prop:main}
 	For $\alpha=1-\sqrt{6}/3$ we have
 	\begin{equation*}
 	\int_{\mathbb{S}^{d-1}} \HH_\infty^{s-\tau}(\vis_\theta(E))\, d\theta \lesssim_{\varepsilon} \delta^{\varepsilon}.
 	\end{equation*}
 \end{prop}
 
 Taking $\delta\to 0$, this will give \eqref{eq:goal}.
 
 Let $\Delta\in 2^{-\mathbb{N}}$ be the dyadic number satisfying
 \begin{equation*}
 \delta^\kappa< \Delta \le 2\delta^\kappa,
 \end{equation*}
 where $\kappa\in(0,1)$ is a constant to be fixed later. In the end we will choose $\kappa = \alpha/(1-\alpha)$. Note that $\Delta \gg \delta$.
 
 Recall that $\mathbb{D}$ stands for the dyadic cubes in $\R^d$. We set
 \begin{align*}
 \DD &\coloneqq \{Q\in\mathbb{D} : Q\cap E\neq\varnothing\},\\
 \DD_\delta &\coloneqq \{Q\in\DD : \ell(Q)=\delta\},\\
 \DD_\Delta &\coloneqq \{Q\in\DD : \ell(Q)=\Delta\}.
 \end{align*}
 Since $E$ is $s$-Ahlfors regular, we have $\#\DD_\delta\sim \delta^{-s}$ and $\#\DD_\Delta\sim \delta^{-\kappa s}$. 
 %	Given $Q\in\DD_\Delta$, we set $\DD_\delta(Q) = \{P\in\DD_\delta \ :\ P\subset 3Q\}$, so that $\#\DD_\delta(Q)\sim (\delta/\Delta)^{-s}\sim \delta^{(1-\kappa)s}$.
 
 For every $Q\in\DD_\Delta$ we define also
 \begin{equation*}
 \mu_Q = \mu|_{3Q}.
 \end{equation*}
 Note that $\mu_Q(\R^d)=\mu(3Q)\sim \ell(Q)^s\sim \delta^{\kappa s}$.
 
 Recall that $\mu_\theta$ denotes the push-forward of $\mu$ by the orthogonal projection $\pi_\theta$, so that it is a measure on $\theta^\perp\simeq \R^{d-1}$.
 
 Let $\sigma = (s-d+1-\varepsilon)/2$. By \lemref{lem:sobolevest} we have 
 \begin{equation}\label{eq:sobolevestmu}
 \int_{\mathbb{S}^{d-1}}\|\mu_\theta\|^2_{H^\sigma(\R^{d-1})}\, d\theta\lesssim_{\varepsilon} 1,
 \end{equation}
 and for every $Q\in\DD_\Delta$
 \begin{equation}\label{eq:sobolevestmuQ}
 \int_{\mathbb{S}^{d-1}}\|\mu_{Q,\theta}\|^2_{H^\sigma(\R^{d-1})}\, d\theta\lesssim_{\varepsilon} \mu(3Q)\sim \delta^{\kappa s}.
 \end{equation}
 
 %	\begin{proof}
 %		By \lemref{lem:energyest} and \lemref{lem:sobolevest} we have the following estimate for the homogeneous Sobolev norm:
 %		\begin{equation*}
 %		\int_{\mathbb{S}^{d-1}}\|\mu_\theta\|_{\dot{H}^\sigma(\R^{d-1})}^2\, d\theta\lesssim I_{s-\varepsilon}(\mu)\lesssim_\varepsilon \mu(\R^d)\sim 1.
 %		\end{equation*}
 %		Together with \eqref{eq:7} this gives the desired estimate of the inhomogeneous Sobolev norm \eqref{eq:sobolevestmu}. The proof of \eqref{eq:sobolevestmuQ} is analogous.
 %	\end{proof}	
 Given $\theta\in\mathbb{S}^{d-1}$ and a dyadic parameter $\gamma\in 2^{-\mathbb{N}}$, let $\T_{\gamma,\theta}$ be a family of (approximate) tubes given by
 \begin{equation*}
 \T_{\gamma,\theta}\coloneqq\{T\ :\ T=\pi_\theta^{-1}(Q),\ Q\in\mathbb{D}_{\R^{d-1}},\  \ell(Q)=\gamma,\ T\cap [0,1]^d\neq\varnothing\},
 \end{equation*}	
 where $\mathbb{D}_{\R^{d-1}}$ denotes the standard dyadic cubes on $\R^{d-1}$. When the direction $\theta$ is clear from context, we will simply write $\T_{\gamma}$.
 
 Note that the tubes from $\T_{\gamma,\theta}$ have width $\sim\gamma$, they are parallel to $\theta$, they cover $[0,1]^d$, and $\#\T_{\gamma,\theta}\sim \gamma^{1-d}$. 
 
 We set $\T_\theta=\bigcup_\gamma\T_{\gamma,\theta}$, and for $T\in\T_\theta$ we denote by $w(T)$ the width of $T$, i.e. the unique $\gamma$ such that $T\in\T_{\gamma}$.
 
 Given a tube $T=\pi_\theta^{-1}(Q)\in \T_{\gamma,\theta}$ and a constant $C>1$, we will write $CT$ to denote $\pi_\theta^{-1}(CQ)$.
 
 \vspace{1em}
 
 Let $\theta\in\mathbb{S}^{d-1}$. We will divide $\vis_\theta(E)$ into four parts. The Hausdorff content of each will be estimated separately. 
 
 \subsection{Heavy tubes}	
 Since $E$ is an $s$-Ahlfors regular set, 
 %	one expects that a typical slice $E\cap\ell$, where $\ell$ is a line, should be at most $s-d+1$ dimensional. More quantitatively, 
 we expect that for a typical tube $T\in\T_\gamma=\T_{\gamma,\theta}$
 \begin{equation*}
 N(T\cap E,\gamma)\lesssim \gamma^{-(s-d+1)}.
 \end{equation*}
 We are going to bound the number of exceptional tubes where the estimate above fails badly.
 
 We say that $T\in\T_\theta$ is \emph{heavy}, denoted by $\T_{H}$, if
 \begin{equation}\label{eq:heavytube}
 N(T\cap E,\gamma)\ge\delta^{-2\varepsilon}\gamma^{-s+d-1},
 \end{equation}
 where $\gamma=w(T)$.
 
 Let $Q\in\DD_\Delta$. We will say that $T\in\T_\delta$ is \emph{heavy inside $Q$}, denoted by $\T_{H,Q}$, if
 \begin{equation}\label{eq:heavytube2}
 N(T\cap E\cap Q,\delta)\ge\delta^{(\kappa-1)(s-d+1)-\kappa\tau - 4\varepsilon}.
 \end{equation}
 Remark that the definition of $\T_H$ includes tubes of varying widths, whereas the tubes in $\T_{H,Q}$ are all $\delta$-tubes.
 
 We define
 \begin{equation*}
 E_{H}' \coloneqq E\cap\bigcup_{T\in\T_H}T,
 \end{equation*}
 and for each $Q\in\DD_\Delta$ we define
 \begin{equation*}
 Q_H\coloneqq 3Q\cap E\cap \bigcup_{T\in\T_{H,Q}}T.
 \end{equation*}	
 Finally, we set
 \begin{equation*}
 E_H \coloneqq E_{H}'\cup \bigcup_{Q\in\DD_\Delta} Q_{H}.
 \end{equation*}
 
 \begin{remark}
 	Note that the definitions above depend on the direction $\theta$. To simplify notation we usually suppress this dependence, but at times we will write $E_{H,\theta}$ instead of $E_H$. The same applies to other sets defined in this section.
 \end{remark}

 \subsection{Light tubes} We will say that a tube $T\in\T_\delta$ is \emph{light} if
 \begin{equation*}
 N(T\cap E,\delta)\le \delta^{-(s-d+1)+\tau+\varepsilon}.
 \end{equation*}
 We denote the family of light tubes by $\T_L$, and we set
 \begin{equation*}
 E_L = E\cap\bigcup_{T\in\T_L}T.
 \end{equation*}
 
 \begin{lemma}\label{lem:lightpart}
 	We have 
 	\begin{equation*}
 	\HH^{s-\tau}_\infty(E_L)\lesssim \delta^{\varepsilon}.
 	\end{equation*}
 \end{lemma}
 \begin{proof}
 	Observe that
 	\begin{equation*}
 	N(E_L,\delta) \le \sum_{T\in\mathcal{T}_L}N(T\cap E,\delta)\lesssim \delta^{1-d}\cdot \delta^{-s+d-1+\tau+\varepsilon}=\delta^{-s+\tau+\varepsilon},
 	\end{equation*}
 	which gives
 	\begin{equation*}
 	\HH_\infty^{s-\tau}(E_L)\le N(E_L,\delta)\cdot \delta^{s-\tau}\lesssim \delta^{\varepsilon}.
 	\end{equation*}
 \end{proof}
 
 \subsection{Good and bad parts}
 Let $Q\in\DD_\Delta$. We will say that a tube $T\in\T_{\delta}$ \emph{substantially intersects $Q$}, denoted by $T\in\T(Q)$, if
 \begin{equation}\label{eq:defsubstinter}
 N(T\cap E\cap Q,\, \delta)\ge \delta^{(\kappa-1)(s-d+1) + \tau+ 4\varepsilon}.
 \end{equation}
 The following lemma explains where the exponent $(\kappa-1)(s-d+1) + \tau+ 4\varepsilon$ came from.
 \begin{lemma}\label{lem:normaltubes}
 	Let $T\in\T_{\delta}$ be a tube which is not contained in any tube from $\T_L\cup\T_H$. Then, there exists $Q\in\DD_{\Delta}$ such that $T\in\T(Q)$.
 \end{lemma}
 \begin{proof}
 	Let $T'\in\T_{\Delta}$ be such that $T\subset T'$. By our assumption, $T'\notin\T_H$, which gives
 	\begin{equation*}
 	N(T\cap E,\Delta)\le N(T'\cap E,\Delta)\le \delta^{-2\varepsilon}\Delta^{-s+d-1}\sim \delta^{-\kappa(s-d+1)-2\varepsilon}.
 	\end{equation*}
 	At the same time, since $T\notin\T_{L}$, we have $N(T\cap E,\delta)> \delta^{-(s-d+1)+\tau+\varepsilon}.$ The two inequalities imply that for some $Q\in\DD_{\Delta}$ we have
 	\begin{equation*}
 	N(T\cap E\cap Q,\delta)\gtrsim \frac{N(T\cap E,\delta)}{N(T\cap E,\Delta)}\gtrsim \delta^{(\kappa-1)(s-d+1)+\tau+3\varepsilon}.
 	\end{equation*}
 \end{proof}
 
 Let $\mathcal{L}$ denote the set of affine lines parallel to $\theta$. We say that a line $\ell\in\cL$ is \emph{bad with respect to $Q\in\DD_{\Delta}$} if the $\delta$-tube $T\in\T_{\delta}$ with $\ell\subset T$ satisfies $T\in\T(Q)$, and at the same time
 %	\begin{equation}\label{eq:defbadline}
 %		N(\ell(2\delta)\cap E\cap 3Q, \delta)\ge \delta^{(\kappa-1)(s-d+1) + \tau+ 2\varepsilon}\quad\text{and}\quad  \ell\cap E\cap \overline{3Q}=\varnothing.
 %	\end{equation}
 %	\begin{equation}\label{eq:defbadline}
 %	\#\{P\in\DD_\delta(Q)\ :\ P\cap\ell(2\delta)\neq\varnothing \}\ge \delta^{(\kappa-1)(s-d+1) + \tau+ 2\varepsilon}\quad\text{and}\quad  \ell\cap E\cap \overline{3Q}=\varnothing.
 %	\end{equation}
 \begin{equation}\label{eq:defbadline}
 \ell\cap E\cap \overline{3Q}=\varnothing.
 \end{equation}
 We denote the collection of bad lines with respect to $Q$ by $\cL_{Q,B}$, and set
 \begin{equation*}
 L_{Q,B} = \bigcup_{\ell\in\cL_{Q,B}} \ell,\quad L_B=\bigcup_{Q\in\DD_{\Delta}}L_{Q,B}.
 \end{equation*}
 We define the bad part of $E$ as
 \begin{equation*}
 E_B \coloneqq E\cap L_B\setminus E_H,
 \end{equation*}
 and the good part of $E$ as
 \begin{equation*}
 E_G \coloneqq E\setminus (E_B\cup E_H\cup E_L).
 \end{equation*}
 \vspace{1em}
 
 It is clear that
 \begin{equation*}
 \vis_\theta(E) \subset E_H \cup E_L \cup E_B\cup (E_G\cap\vis_\theta(E)).
 \end{equation*}
 In order to estimate the Hausdorff content of $\vis_\theta(E)$, we will estimate the contents of the four sets on the right hand side separately. We have already estimated $\HH^{s-\tau}_\infty(E_L)$ in \lemref{lem:lightpart}. The estimates for $\HH^{s-\tau}_\infty(E_H),$ $\HH^{s-\tau}_\infty(E_B)$, and $\HH^{s-\tau}_\infty(E_G\cap\vis_\theta(E))$ are obtained in the next three sections.
 
 \section{Heavy part}
 Recall that $\cM$ denotes the Hardy-Littlewood maximal function. We define
 \begin{equation*}
 \widetilde{E}_{H}' \coloneqq \{x\in E\ :\ \mathcal{M}\mu_\theta(\pi_\theta(x))\ge \delta^{-\varepsilon}  \},
 \end{equation*}
 and for each $Q\in\DD_\Delta$ we define
 \begin{equation*}
 \widetilde{Q}_H\coloneqq \{ x\in 3Q\cap E\ :\ \mathcal{M}\mu_{Q,\theta}(\pi_\theta(x))\ge \delta^{\kappa (s-d+1)-\kappa\tau-3\varepsilon}  \}.
 \end{equation*}	
 \begin{lemma}\label{lem:heavytube}
 	We have $E_H'\subset \widetilde{E}_H'$ and $Q_H\subset\widetilde{Q}_H$.
 \end{lemma}
 \begin{proof}
 	Let $x\in E_H'$, so that $x\in T\cap E$ for some $T\in\T$ satisfying \eqref{eq:heavytube}. It follows that for some absolute $C>1$
 	\begin{multline*}
 	\mathcal{M}\mu_\theta(\pi_\theta(x))\ge \frac{\mu_\theta(B(\pi_\theta(x), Cw(T)))}{(Cw(T))^{d-1}}\gtrsim w(T)^{-d+1}\mu(2T)\\
 	\gtrsim w(T)^{-d+1} N(T\cap E,w(T))\cdot w(T)^{s}\overset{\eqref{eq:heavytube}}{\ge} \delta^{-2\varepsilon},
 	\end{multline*}
 	where in the third inequality we used $s$-Ahlfors regularity of $E$. Hence, $x\in \widetilde{E}_H'$.
 	
 	Similarly, if $T\in\T_\delta$ satisfies \eqref{eq:heavytube2}, then for $x\in T\cap Q\cap E$
 	\begin{equation*}
 	\mathcal{M}\mu_{Q,\theta}(\pi_\theta(x))\gtrsim\delta^{-d+1}\mu_Q(2T)\gtrsim \delta^{-d+1} N(T\cap Q\cap E,\delta)\cdot \delta^{s}\ge \delta^{\kappa (s-d+1)-\kappa\tau-4\varepsilon},
 	\end{equation*}
 	which gives $x\in \widetilde{Q}_H$.
 \end{proof}
 
 %	In order to estimate $\HH_\infty^{s-\tau}(\widetilde{E}'_{H})$ and $\sum_{Q\in\DD_{\Delta}}\HH_\infty^{s-\tau}(\widetilde{Q}_{H})$, we first prove the following result.
 We use \propref{prop:heavypart} to estimate the size of $E_H$.
 \begin{lemma}\label{lem:heavyest}
 	We have
 	\begin{equation}\label{eq:heavyest}
 	\int_{\mathbb{S}^{d-1}}\HH_\infty^{s-\tau}(E_{H,\theta})\ d\theta\lesssim_{\varepsilon} \delta^{\varepsilon}.
 	\end{equation}
 \end{lemma}
 \begin{proof}
 	Recall that 
 	\begin{equation*}
 	E_{H,\theta} = E_{H,\theta}' \cup \bigcup_{Q\in\DD_\Delta}Q_{H,\theta}.
 	\end{equation*} 		
 	By \lemref{lem:heavytube} and \propref{prop:heavypart} applied with $n=d-1$, $\nu=\mu,\ F=E$ and $M=\delta^{-\varepsilon}$, we have
 	\begin{equation*}
 	\HH_\infty^{d-1+2\varepsilon}(E_{H,\theta}')\le\HH_\infty^{d-1+2\varepsilon}(\widetilde{E}_{H,\theta}')\lesssim_\varepsilon\delta^\varepsilon\|\mu_\theta\|_{H^\sigma(\R^{d-1})}^2,
 	\end{equation*}
 	so that
 	\begin{equation}\label{eq:8}
 	\int_{\mathbb{S}^{d-1}}\HH_\infty^{d-1+2\varepsilon}(E_{H,\theta}')\ d\theta\lesssim_\varepsilon \delta^{\varepsilon}\int_{\mathbb{S}^{d-1}}\|\mu_\theta\|_{H^\sigma(\R^{d-1})}^2\, d\theta\overset{\eqref{eq:sobolevestmu}}{\lesssim_{\varepsilon}} \delta^{\varepsilon}.
 	\end{equation}
 	Since $d-1+2\varepsilon<s-\tau$, this is even better than the estimate for $\HH_\infty^{s-\tau}(E_{H,\theta}')$ that we need thanks to the simple inequality
 	\begin{equation}\label{eq:simpleineq}
 	\HH^a_\infty(A)\le \diam(A)^{a-b} \cdot \HH^b_\infty(A)\quad \text{if $a>b$}.
 	\end{equation}
 	
 	We move on to estimating $\HH_\infty^{s-\tau}(Q_{H,\theta})$. We apply again \propref{prop:heavypart} with $n=d-1$, $\nu=\mu_Q,$ $F=3\overline{Q}\cap E$ and $M=\delta^{\kappa(s-d+1)-\kappa\tau-3\varepsilon}$ to get
 	\begin{equation*}
 	\HH_\infty^{d-1+2\varepsilon}(Q_{H,\theta})\le \HH_\infty^{d-1+2\varepsilon}(\widetilde{Q}_{H,\theta})\lesssim_\varepsilon\delta^{\kappa\tau+3\varepsilon-\kappa(s-d+1)}\|\mu_{Q,\theta}\|_{H^\sigma(\R^{d-1})}^2.
 	\end{equation*}
 	We use \eqref{eq:simpleineq} to estimate
 	\begin{multline*}
 	\HH_\infty^{s-\tau}(Q_{H,\theta})\lesssim \ell(Q)^{s-\tau - (d-1+2\varepsilon)}\,\HH_\infty^{d-1+2\varepsilon}(Q_{H,\theta})\\
 	\lesssim_\varepsilon\delta^{\kappa(s-d+1-\tau-2\varepsilon) -\kappa(s-d+1-\tau)+3\varepsilon}\|\mu_{Q,\theta}\|_{H^\sigma(\R^{d-1})}^2\\
 	=\delta^{(3-2\kappa)\varepsilon}\|\mu_{Q,\theta}\|_{H^\sigma(\R^{d-1})}^2\le \delta^\varepsilon\|\mu_{Q,\theta}\|_{H^\sigma(\R^{d-1})}^2.
 	\end{multline*}
 	Summing over $Q\in\DD_\Delta$ and integrating over $\theta\in\mathbb{S}^{d-1}$ yields
 	\begin{multline*}
 	\int_{\mathbb{S}^{d-1}}\sum_{Q\in\DD_\Delta}\HH_\infty^{s-\tau}(Q_{H,\theta})\ d\theta\lesssim_\varepsilon \delta^\varepsilon\sum_{Q\in\DD_\Delta}\int_{\mathbb{S}^{d-1}}\|\mu_{Q,\theta}\|_{H^\sigma(\R^{d-1})}^2\, d\theta\\
 	\overset{\eqref{eq:sobolevestmuQ}}{\lesssim_\varepsilon} \delta^\varepsilon\sum_{Q\in\DD_\Delta} \mu(3Q) \sim \delta^\varepsilon.
 	\end{multline*}
 	Together with \eqref{eq:8}, this finishes the proof of \eqref{eq:heavyest}.
 \end{proof}

 \section{Bad part}
 Fix $\theta\in\mathbb{S}^{d-1}$. Recall that the collection of bad lines with respect to $Q$ parallel to $\theta$ is denoted by $\cL_{Q,B}$,	$L_{Q,B} = \bigcup_{\ell\in\cL_{Q,B}} \ell$, $L_B=\bigcup_{Q\in\DD_{\Delta}}L_{Q,B}$, and that the bad part was defined as $E_{B} = E\cap L_{B}\setminus E_{H}$. In this section we estimate $\HH_\infty^{s-\tau}(E_B)$.
 
 %	For every $Q\in\Delta$ let $\cL_{Q,B}\subset\cL_B$ denote the collection of bad lines for which the bad line property \eqref{eq:defbadline} holds with $Q$. Clearly, $\cL_B=\bigcup_{Q\in\DD_\Delta}\cL_{Q,B}$. 
 %	Let
 %	\begin{equation*}
 %		L_{Q,B} = \bigcup_{\ell\in\cL_{Q,B}}\ell,
 %	\end{equation*}
 %	so that $L_B=\bigcup_{Q\in\DD_\Delta}L_{Q,B}$.
 \begin{lemma}\label{lem:contentprojection}
 	Let $\sigma=(s-d+1-\varepsilon)/2$. For every $Q\in\DD_{\Delta}$ we have
 	\begin{equation*}
 	\HH_\infty^{d-1-\tau}(\pi_\theta(L_{Q,B})) \lesssim_\varepsilon \delta^{(1-3\alpha-2\kappa)(s-d+1)+5\varepsilon}\|\mu_{Q,\theta}\|_{H^{\sigma}(\R^{d-1})}^2.
 	\end{equation*}
 \end{lemma}
 \begin{proof}
 	Suppose that $\HH_\infty^{d-1-\tau}(\pi_\theta(L_{Q,B}))>0$, otherwise there is nothing to prove. It is easy to see that $\pi_\theta(L_{Q,B})\subset\R^{d-1}$ is a Borel set, and so we can use the usual Frostman's lemma (\cite[Theorem 8.8]{mattila1999geometry}) to get  a Borel probability measure $\nu$ with $\supp\nu\subset \pi_\theta(L_{Q,B})$ and
 	\begin{equation}\label{eq:nuFrostman}
 	\nu(B(x,r))\lesssim \HH_\infty^{d-1-\tau}(\pi_\theta(L_{Q,B}))^{-1}r^{d-1-\tau}.
 	\end{equation}
 	
 	By the definition of bad lines in $\cL_{Q,B}$ \eqref{eq:defbadline}, we have 
 	\begin{equation*}
 	\pi_\theta(L_{Q,B})\cap \pi_\theta(\overline{3Q}\cap E)=\varnothing.
 	\end{equation*}
 	This means that $\supp\nu\cap\supp\mu_{Q,\theta}=\varnothing$. Since the supports are compact, we get that for $\eta>0$ small enough
 	\begin{equation}\label{eq:integralzero}
 	\int \mu_{Q,\theta}\ast\varphi_{\eta}\ d\nu=0,
 	\end{equation}
 	where $\varphi_{\eta}$ is a smooth mollifier, as in the proof of \lemref{lem:heavyaux}. At the same time, using Plancherel's identity
 	\begin{multline*}
 	\int \mu_{Q,\theta}\ast\varphi_{\eta}\ d\nu=\int_{\R^{d-1}}\hat{\varphi}(\eta\xi)\widehat{\mu_{Q,\theta}}(\xi)\overline{\hat{\nu}(\xi)}\ d\xi\\
 	\ge \bigg| \int_{\R^{d-1}} \hat{\varphi}(C\delta\xi)\hat{\varphi}(\eta\xi)\widehat{\mu_{Q,\theta}}(\xi)\overline{\hat{\nu}(\xi)}\ d\xi\bigg|\\
 	- \bigg| \int_{\R^{d-1}} \big(1-\hat{\varphi}(C\delta\xi)\big)\hat{\varphi}(\eta\xi)\widehat{\mu_{Q,\theta}}(\xi)\overline{\hat{\nu}(\xi)}\ d\xi\bigg| = I_1 - I_2,
 	\end{multline*}
 	where $C\ge 1$ is an absolute constant to be fixed below. Note that by \eqref{eq:integralzero} we have $I_1= I_2$.
 	
 	We estimate $I_1$ from below. Note that, by Plancherel,
 	\begin{equation*}
 	I_1 = \int \mu_{Q,\theta}\ast\varphi_{C\delta}\ast\varphi_{\eta}\ d\nu.
 	\end{equation*}
 	Observe that, by the definition of bad lines, and \eqref{eq:defsubstinter}, for any $x\in \supp\nu\subset \pi_\theta(L_{Q,B})$ the line $\ell=\pi_\theta^{-1}(x)\in\cL_{Q,B}$ satisfies $N(\ell(C\delta)\cap E\cap Q,\delta)\gtrsim \delta^{(\kappa-1)(s-d+1) + \tau+4\varepsilon}$. Recalling that for $y\in E\cap Q$ we have $\mu_Q(B(x,\delta))\sim \delta^{s}$, it follows that
 	\begin{equation*}
 	\mu_{Q}(\ell(C\delta))\gtrsim \delta^{(\kappa-1)(s-d+1) + \tau+4\varepsilon}\cdot \delta^s,
 	\end{equation*}
 	which means that, assuming $C\sim 1$ large enough and $\eta>0$ small enough,
 	\begin{equation*}
 	\mu_{Q,\theta}\ast\varphi_{C\delta}\ast\varphi_{\eta}(x)\gtrsim \delta^{\kappa(s-d+1) + \tau+ 4\varepsilon}.
 	\end{equation*}
 	Since $\nu$ is a probability measure, we get that
 	\begin{equation}\label{eq:I1est}
 	I_1\gtrsim \delta^{\kappa(s-d+1) + \tau+ 4\varepsilon}.
 	\end{equation}
 	
 	%		Taking $C\ge 1$ large enough we have
 	%		\begin{equation*}
 	%		I_1 = \int \mu_{Q,\theta}\ast\varphi_{C\delta}\ast\varphi_{\eta}\ d\nu\gtrsim \delta^{(\alpha + \kappa-1)(s-d+1) + 2\varepsilon}\cdot \delta^{s-d+1} = \delta^{(\alpha + \kappa)(s-d+1) + 2\varepsilon}.
 	%		\end{equation*}
 	Now we estimate $I_2$ from above. Since $\hat{\varphi}$ is a bounded Lipschitz function, and $\hat{\varphi}(0)=1$, we have for any $0\le a\le 1$
 	\begin{equation*}
 	|1-\hat{\varphi}(C\delta\xi)|=|\hat{\varphi}(0)-\hat{\varphi}(C\delta\xi)|\lesssim \min(|\delta\xi|,1)\le\delta^a|\xi|^{a}.
 	\end{equation*}
 	Thus, by the Cauchy-Schwarz inequality
 	\begin{multline*}
 	I_2\lesssim \delta^{a}\int_{\R^{d-1}}|\xi|^a |\widehat{\mu_{Q,\theta}}(\xi)||\hat{\nu}(\xi)|\ d\xi\\
 	\le \delta^{a}\left(\int_{\R^{d-1}} |\widehat{\mu_{Q,\theta}}(\xi)|^{2}|\xi|^{s-d+1-\varepsilon}\ d\xi  \right)^{1/2}\left(\int_{\R^{d-1}} |\widehat{\nu}(\xi)|^{2}|\xi|^{(2a - s + 2(d-1)+\varepsilon) - (d-1)}\ d\xi  \right)^{1/2}.
 	\end{multline*}
 	Since $\nu$ satisfies the Frostman condition \eqref{eq:nuFrostman}, it follows from the Riesz energy estimate \eqref{eq:rieszener} and the identity \eqref{eq:energysobolev} that if $2a - s + 2(d-1) +\varepsilon\le d-1-\tau-\varepsilon$, then
 	\begin{equation*}
 	\int_{\R^{d-1}} |\widehat{\nu}(\xi)|^{2}|\xi|^{(2a - s + 2(d-1)+\varepsilon) - (d-1)}\ d\xi\lesssim_{\varepsilon}  \HH_\infty^{d-1-\tau}(\pi_\theta(L_{Q,B}))^{-1}.
 	\end{equation*}
 	Together with the preceding estimate, we arrive at
 	\begin{equation*}
 	I_2\lesssim_{\varepsilon} \delta^a  \|\mu_{Q,\theta}\|_{H^{\sigma}(\R^{d-1})}\, \HH_\infty^{d-1-\tau}(\pi_\theta(L_{Q,B}))^{-1/2},
 	\end{equation*}
 	as long as 
 	\begin{equation*}
 	0\le a \le \frac{(s - d+1)-\tau}{2} -\varepsilon.
 	\end{equation*}		
 	
 	Recalling that $I_1= I_2$, we compare the estimate of $I_2$ with \eqref{eq:I1est} and obtain
 	\begin{equation*}
 	\delta^a  \|\mu_{Q,\theta}\|_{H^{\sigma}(\R^{d-1})}\, \HH_\infty^{d-1-\tau}(\pi_\theta(L_{Q,B}))^{-1/2} \gtrsim_{\varepsilon} \delta^{\kappa(s-d+1)+\tau+4\varepsilon},
 	\end{equation*}
 	which is equivalent to
 	\begin{equation*}
 	\HH_\infty^{d-1-\tau}(\pi_\theta(L_{Q,B})) \lesssim_\varepsilon \delta^{2a - 2\kappa(s-d+1)-2\tau-8\varepsilon}\|\mu_{Q,\theta}\|_{H^{\sigma}(\R^{d-1})}^2.
 	\end{equation*}
 	Choosing $a=\frac{(s - d+1)-\tau}{2} -\varepsilon$ we get
 	\begin{multline*}
 	\HH_\infty^{d-1-\tau}(\pi_\theta(L_{Q,B})) \lesssim_\varepsilon \delta^{(1-2\kappa)(s-d+1)-3\tau-10\varepsilon}\|\mu_{Q,\theta}\|_{H^{\sigma}(\R^{d-1})}^2\\
 	=\delta^{(1-2\kappa-3\alpha)(s-d+1)+5\varepsilon}\|\mu_{Q,\theta}\|_{H^{\sigma}(\R^{d-1})}^2,
 	\end{multline*}
 	where we used $\tau=\alpha(s-d+1)-5\varepsilon$.	
 	%		Our second constrain, which we need for $I_1\gg I_2$, is
 	%		\begin{equation*}
 	%		(\alpha + \kappa)(s-d+1) + \varepsilon < a - \varepsilon-\kappa s/2,
 	%		\end{equation*}
 	%		which is equivalent to
 	%		\begin{equation*}
 	%		a > (\alpha + \kappa)(s-d+1) +2\varepsilon + \kappa s/2.
 	%		\end{equation*}
 \end{proof}
 %	We assume from now on that the parameters $\kappa$ and $\alpha$ satisfy 
 %	\begin{equation}\label{eq:kappaalpharestr}
 %		2\kappa+3\alpha\le 1.
 %	\end{equation}

 \begin{lemma}\label{lem:badest}
 	If $2\kappa+3\alpha\le 1$, then
 	\begin{equation}\label{eq:badestimate}
 	\int_{\mathbb{S}^{d-1}}\HH^{s-\tau}_\infty(E_{B,\theta})\ d\theta\lesssim_\varepsilon\delta^{\varepsilon}.
 	\end{equation}
 \end{lemma}
 \begin{proof}
 	Note that if $\kappa$ and $\alpha$ satisfy $2\kappa+3\alpha\le 1$, then \lemref{lem:contentprojection} gives
 	\begin{equation}\label{eq:10}
 	\HH_\infty^{d-1-\tau}(\pi_\theta(L_{Q,B})) \lesssim_\varepsilon \delta^{5\varepsilon}\|\mu_{Q,\theta}\|_{H^{\sigma}(\R^{d-1})}^2.
 	\end{equation}
 	
 	Let $\cP$ be a family of dyadic cubes in $\R^{d-1}$ that covers $\pi_\theta(L_B)=\bigcup_{Q\in\DD_\Delta} \pi_\theta(L_{Q,B})$ and such that
 	\begin{equation*}
 	\sum_{P\in\cP}\ell(P)^{d-1-\tau} \lesssim \sum_{Q\in\DD_\Delta}\HH_\infty^{d-1-\tau}(\pi_\theta(L_{Q,B})).
 	\end{equation*}
 	Clearly, $\cP$ is a covering of $\pi_\theta(E_{B,\theta}),$ and without loss of generality we may assume that each $P\in\cP$ intersects $\pi_\theta(E_{B,\theta}).$
 	
 	For each $P\in\cP$ let $T_P\in\T$ be the tube parallel to $\theta$ such that $\pi_\theta(T_P)= P$. Clearly, 
 	\begin{equation*}
 	E_{B,\theta}\subset \bigcup_{P\in\cP}T_P\cap E.
 	\end{equation*}
 	Since $P\cap\pi_\theta(E_{B,\theta})\neq\varnothing$, we have $T_P\cap E_{B,\theta}\neq\varnothing$. Recalling that $E_{B,\theta}\cap E_{H,\theta}=\varnothing$, we see that $T_P\cap E\not\subset E_{H,\theta}$. This means that $T_P\notin\T_H$, and so by the definition of heavy tubes \eqref{eq:heavytube},
 	\begin{equation*}
 	N(T_P\cap E,\,\ell(P))\le \delta^{-2\varepsilon}\ell(P)^{-s+d-1}.
 	\end{equation*}
 	
 	%		We refine the covering a little: let $\cP'\subset\cP$ be the subfamily of cubes $P$ such that
 	%		\begin{equation}\label{eq:maximalinP}
 	%		P\cap \{t\in\R^{d-1}\ :\ \cM\mu_\theta(t)\le \delta^{-\varepsilon}\}\neq\varnothing.
 	%		\end{equation}
 	%		Recalling that $E_{B,\theta}\cap E_{H,\theta}=\varnothing$, we see that $\cP'$ is still a cover of $\pi_\theta(E_{B,\theta}).$
 	%		
 	%		By \eqref{eq:maximalinP}, for each $P\in\cP'$ we have
 	%		\begin{equation*}
 	%		\mu_\theta(100P)\lesssim \delta^{-\varepsilon}\ell(P)^{d-1}.
 	%		\end{equation*}
 	%		For every $P\in\cP'$ let $\DD_P\subset\mathbb{D}$ be the family of dyadic cubes $Q\subset [0,1]^d$ with $\ell(Q)=\ell(P),\ \pi_\theta(Q)\cap P\neq\varnothing,$ and $E\cap Q\neq\varnothing$. Then, by the estimate above,
 	%		\begin{equation*}
 	%		\#\mathcal{D}_P \lesssim \delta^{-\varepsilon}\ell(P)^{d-1-s}.
 	%		\end{equation*}
 	It follows that
 	\begin{multline*}
 	\HH^{s-\tau}_\infty(E_{B,\theta})\lesssim \sum_{P\in\cP}N(T_P\cap E,\,\ell(P))\cdot\ell(P)^{s-\tau}\\
 	\lesssim\delta^{-2\varepsilon} \sum_{P\in\cP} \ell(P)^{d-1-\tau}\lesssim \delta^{-2\varepsilon}\sum_{Q\in\DD_\Delta}\HH_\infty^{d-1-\tau}(\pi_\theta(L_{Q,B})).
 	\end{multline*}
 	Together with \eqref{eq:10} this gives
 	\begin{equation*}
 	\HH^{s-\tau}_\infty(E_{B,\theta})\lesssim \delta^{\varepsilon}\sum_{Q\in\DD_\Delta}\|\mu_{Q,\theta}\|_{H^{\sigma}(\R^{d-1})}^2.
 	\end{equation*}		
 	Integrating the above over $\theta\in\mathbb{S}^{d-1}$ and using \eqref{eq:sobolevestmuQ} we get the desired estimate:
 	\begin{equation}\label{eq:badestimate2}
 	\int_{\mathbb{S}^{d-1}}\HH^{s-\tau}_\infty(E_{B,\theta})\ d\theta \lesssim\delta^{\varepsilon}\sum_{Q\in\DD_\Delta}\int_{\mathbb{S}^{d-1}}\|\mu_{Q,\theta}\|_{H^{\sigma}(\R^{d-1})}^2\ d\theta
 	\lesssim_\varepsilon \delta^{\varepsilon}\sum_{Q\in\DD_\Delta} \mu(3Q) \sim \delta^{\varepsilon}.
 	\end{equation}
 \end{proof}
 \section{Good part}
 In this section we estimate $\HH_\infty^{s-\tau}(\vis_\theta(E)\cap E_G)$. Recall that $\T_\delta$ denotes a collection of $\sim \delta^{1-d}$ tubes of width $\delta$, parallel to $\theta$, which cover $[0,1]^d$.
 \begin{lemma}\label{lem:coveringgood}
 	For each $T\in\mathcal{T}_\delta$ we have
 	\begin{equation*}
 	N(\vis_\theta(E)\cap E_G\cap T,\delta)\lesssim \delta^{(\alpha - 1)(s-d+1)-3\varepsilon} + \delta^{(\kappa - 1-\kappa\alpha)(s-d+1)-4\varepsilon}.
 	\end{equation*}
 \end{lemma}
 \begin{proof}
 	If $T\cap E_G=\varnothing$, then there is nothing to prove, so assume that $T\cap E_G\neq\varnothing$. Since $E_L\cap E_G=\varnothing$ and $E_H\cap E_G=\varnothing$, it follows that $T$ is not contained in any tube from $\T_L\cup\T_H$.
 	Hence, \lemref{lem:normaltubes} implies that there exists at least one $Q\in\DD_\Delta$ such that $T\in\T(Q)$, which means that
 	\begin{equation}\label{eq:highstack}
 	N(T\cap E\cap Q,\, \delta)\ge \delta^{(\kappa-1)(s-d+1) + \tau+ 4\varepsilon}.
 	\end{equation}
 	Furthermore, since $T$ is not contained in any tube from $\T_H$, we have that
 	the family $\DD_{\Delta,T}\coloneqq\{Q\in\DD_\Delta\ :\ Q\cap T\neq\varnothing\}$ satisfies
 	\begin{equation}\label{eq:fewbigcubes}
 	\#\DD_{\Delta,T}\lesssim \delta^{-\kappa (s-d+1)-2\varepsilon}.
 	\end{equation}
 	
 	Let $Q_T\in\DD_{\Delta,T}$ be the ``$\theta$-highest'' of all cubes $Q\in\DD_\Delta$ satisfying \eqref{eq:highstack}, in the sense that it maximizes $\inf\{x\cdot\theta : x\in 3Q\}$ among all such cubes. 
 	
 	We partition the family $\DD_{\Delta,T}$ into 3 subfamilies:
 	\begin{align}
 	\DD^<_{\Delta,T}&\coloneqq\big\{Q\in\DD_{\Delta,T}\ :\ \sup_{x\in 3Q}\, x\cdot\theta <  \inf_{x\in 3Q_T}x\cdot\theta\big\},\notag\\
 	\DD^=_{\Delta,T}&\coloneqq\big\{Q\in\DD_{\Delta,T}\ :\ \sup_{x\in 3Q}\, x\cdot\theta \ge  \inf_{x\in 3Q_T}x\cdot\theta \ge  \inf_{x\in 3Q}x\cdot\theta \big\},\label{eq:13}\\
 	\DD^>_{\Delta,T}&\coloneqq\big\{Q\in\DD_{\Delta,T}\ :\ \inf_{x\in 3Q}\, x\cdot\theta >  \inf_{x\in 3Q_T}x\cdot\theta\big\}\notag.
 	\end{align}
 	Roughly speaking, the cubes in $\DD^<_{\Delta,T}$ lie ``$\theta$-below'' $Q_T$, the cubes in $\DD^>_{\Delta,T}$ are ``$\theta$-above'' $Q_T$, and the cubes in $\DD^=_{\Delta,T}$ are ``$\theta$-on-the-same-height'' as $Q_T$.
 	Observe that
 	\begin{gather}
 	\#\DD^=_{\Delta,T}\lesssim 1,\label{eq:few=}\\
 	\#\DD^>_{\Delta,T}\le \#\DD_{\Delta,T}\overset{\eqref{eq:fewbigcubes}}{\lesssim}\delta^{-\kappa (s-d+1)-2\varepsilon}.\label{eq:few>}
 	\end{gather}
 	We will estimate $N(\vis_\theta(E)\cap E_G\cap T\cap Q,\delta)$ for cubes $Q$ in different families separately.
 	\vspace{1em}
 	
 	\emph{Cubes from $\DD^<_{\Delta,T}$}. We claim that for $Q\in \DD^<_{\Delta,T}$ we have $\vis_\theta(E)\cap E_G\cap T\cap Q=\varnothing$, so that
 	\begin{equation}
 	N(\vis_\theta(E)\cap E_G\cap T\cap Q,\delta)=0\quad\text{for $Q\in \DD^<_{\Delta,T}$.}
 	\end{equation}
 	
 	Indeed, suppose that there is some $x\in \vis_\theta(E)\cap E_G\cap T\cap Q$. Let $\ell\in\cL$ be the affine line such that $x\in \ell\subset T$. Since $x\in E_G$, we have $x\notin L_B$, and in particular $\ell\notin\cL_{Q_T,B}$, i.e. $\ell$ is not bad with respect to $Q_T$. Recalling that $T\in\T(Q_T)$, it follows from the definition of bad lines \eqref{eq:defbadline} that $\ell\cap E\cap \overline{3Q_T}\neq\varnothing.$
 	%		Observe that $T\subset \ell(2\delta)$, and so by \eqref{eq:highstack}, 
 	%		\begin{equation*}
 	%		\#\{P\in\DD_\delta(Q_T)\ :\ P\cap \ell(2\delta)\neq\varnothing  \} \ge \delta^{(\alpha + \kappa-1)(s-d+1) + 2\varepsilon}.
 	%		\end{equation*}

 	Let $y\in \ell\cap E\cap \overline{3Q_T}$. Then, by the definition of $\DD^<_{\Delta,T}$ we have $x\cdot \theta <y\cdot\theta$. But the assumption $x\in\vis_\theta(E)$ implies that for all $z\in \ell\cap E$ we have $x\cdot \theta > z\cdot\theta$, by the definition of $\vis_\theta(E)$. We have reached a contradiction.
 	
 	\vspace{1em}
 	
 	\emph{Cubes from $\DD^=_{\Delta,T}$}. Recall that $E_G\cap E_H=\varnothing$. In particular, for any $Q\in \DD^=_{\Delta,T}$ we have $E_G \cap Q_H = \varnothing$. There are two possibilities: either $T\in\T_{H,Q}$, or $T\notin\T_{H,Q}$. In the first case, we get $Q\cap E\cap T\subset Q_H\subset E_H$, so that $Q\cap E_G\cap T=\varnothing$. In the latter case, the definition of $\T_{H,Q}$ \eqref{eq:heavytube2} gives
 	\begin{equation}
 	N(Q\cap E_G\cap T,\delta)\lesssim \delta^{(\kappa-1)(s-d+1)-\kappa\tau - 4\varepsilon}\le \delta^{(\kappa-1-\kappa\alpha)(s-d+1)- 4\varepsilon}.
 	\end{equation}

 	\vspace{1em}
 	
 	\emph{Cubes from $\DD^>_{\Delta,T}$}. Recall that $Q_T$ was defined as the ``$\theta$-highest'' cube in $\DD_{\Delta,T}$ satisfying \eqref{eq:highstack}. This means that all the cubes in $\DD^>_{\Delta,T}$ do not satisfy \eqref{eq:highstack}, and so for $Q\in \DD^>_{\Delta,T}$
 	\begin{equation}
 	N(T\cap E\cap Q,\delta)\lesssim \delta^{(\alpha + \kappa-1)(s-d+1) -\varepsilon}.
 	\end{equation}
 	
 	\vspace{1em}
 	
 	We use the three estimates above to conclude that
 	\begin{multline*}
 	N(T\cap E_G\cap\vis_\theta(E),\delta)\le \sum_{Q\in\DD_{\Delta,T}} N(Q\cap T\cap E_G\cap\vis_\theta(E),\delta)\\
 	\lesssim \sum_{Q\in\DD_{\Delta,T}^=} \delta^{(\kappa-1-\kappa\alpha)(s-d+1)- 4\varepsilon} + \sum_{Q\in\DD_{\Delta,T}^>}\delta^{(\alpha + \kappa-1)(s-d+1) -\varepsilon}
 	\\ \overset{\eqref{eq:few=},\eqref{eq:few>}}{\lesssim} \delta^{(\kappa - 1-\kappa\alpha)(s-d+1)-4\varepsilon}  + \delta^{-\kappa(s-d+1)-2\varepsilon} \cdot \delta^{(\alpha + \kappa-1)(s-d+1) -\varepsilon} \\
 	= \delta^{(\kappa - 1-\kappa\alpha)(s-d+1)-4\varepsilon}  + \delta^{(\alpha -1)(s-d+1) -3\varepsilon}
 	\end{multline*}
 \end{proof}
 
 \begin{lemma}\label{lem:goodpartest}
 	If $\kappa=\alpha/(1-\alpha)$ and $\alpha\in (0,1/2)$, then 
 	\begin{equation*}
 	\HH_\infty^{s-\tau}(\vis_\theta(E)\cap E_G)\lesssim\delta^\varepsilon.
 	\end{equation*}
 \end{lemma}
 \begin{proof}
 	Using the estimate from \lemref{lem:coveringgood} we get that
 	\begin{multline*}
 	\HH_\infty^{s-\tau}(\vis_\theta(E)\cap E_G)\le \sum_{T\in\mathcal{T}_\delta} N(\vis_\theta(E)\cap E_G\cap T,\delta)\cdot\delta^{s-\tau} \\
 	\lesssim \delta^{-d+1}\cdot ( \delta^{(\kappa - 1-\kappa\alpha)(s-d+1)-4\varepsilon}  + \delta^{(\alpha -1)(s-d+1)-3\varepsilon})\cdot\delta^{s-\tau}\\
 	= \delta^{(\kappa -\kappa\alpha-\alpha)(s-d+1)+\varepsilon}  + \delta^{ 2\varepsilon}
 	\end{multline*}
 	Taking $\kappa=\alpha/(1-\alpha),$ which satisfies $\kappa\in (0,1)$ for $\alpha\in (0,1/2)$, we obtain the desired inequality.
 \end{proof}
 
 \subsection{Conclusion of the proof of \thmref{thm:main}}
 Putting together the estimates from \lemref{lem:lightpart}, \lemref{lem:heavyest}, \lemref{lem:badest}, and \lemref{lem:goodpartest}, we get that
 \begin{equation*}
 \int_{\mathbb{S}^{d-1}} \HH_\infty^{s-\tau}(\vis_\theta(E))\, d\theta \lesssim_{\varepsilon} \delta^{\varepsilon}
 \end{equation*}
 as long as $2\kappa+3\alpha\le 1,$ $\kappa=\alpha/(1-\alpha)$, and $\alpha\in (0,1/2)$.
 
 Plugging $\kappa=\alpha/(1-\alpha)$ into $2\kappa+3\alpha\le 1$ we arrive at
 \begin{equation*}
 -3\alpha^2+6\alpha-1\le 0,
 \end{equation*} 
 and the largest $\alpha\in(0,1/2)$ which satisfies this inequality is $\alpha=1-\sqrt{6}/3$. This gives \propref{prop:main}. Taking $\delta\to 0$, and then $\varepsilon\to 0$, \thmref{thm:main} follows.
 
 \section{General compact sets}
 In this section we prove \thmref{thm:main2}. Since the proof follows quite closely that of \cite{orponen2022visible}, with the major changes already present in our proof of \thmref{thm:main}, we will be brief. 
 
 Let $E\subset \R^d$ be compact. By rescaling, we may assume that $E\subset [0,1)^d$.	
 Let $\varepsilon>0$ be a small constant, and set
 \begin{equation*}
 \tau \coloneqq \frac{1}{6}-5\varepsilon.
 \end{equation*}
 Our goal is to prove that
 \begin{equation*}
 \HH_\infty^{d-\tau}(\vis_\theta(E))=0\quad\text{for a.e. $\theta\in\TT$}.
 \end{equation*}
 By taking $\varepsilon\to 0$, \eqref{eq:dimest2} will follow.
 \subsection{Preliminaries}
 Fix a small dyadic scale $\delta\in 2^{-\mathbb{N}}$, and let $\Delta\in 2^{-\mathbb{N}}$ with $\Delta\sim \delta^{\kappa}$, where $\kappa=1/6$, so that $\Delta\gg\delta$. 
 We will show that 
 \begin{equation*}
 \int_{\mathbb{S}^{d-1}} \HH_\infty^{d-\tau}(\vis_\theta(E))\, d\theta \lesssim_{\varepsilon} \delta^{\varepsilon}.
 \end{equation*}
 
 As before, we set $\DD = \{Q\in\mathbb{D} : Q\cap E\neq\varnothing\},$
 $\DD_\delta = \{Q\in\DD : \ell(Q)=\delta\},$ and
 $\DD_\Delta = \{Q\in\DD : \ell(Q)=\Delta\}.$ 
 
 Let $\mu$ be the Frostman measure on $E$ given by \lemref{lem:Frostman}, with exponent $t=d-\tau$, so that
 \begin{equation}\label{eq:Frostupper}
 \mu(B(x,r))\lesssim r^{d-\tau}
 \end{equation}
 and for all $Q\in\DD_\delta$
 \begin{equation}\label{eq:Frostlower}
 \mu(\overline{Q})\gtrsim\min(\HH_\infty^{d-\tau}(E\cap Q),\ell(Q)^d).
 \end{equation}
 For every $Q\in\DD_{\Delta}$ we define $\mu_Q=\mu|_Q$. Let $\sigma=(1-\tau-\varepsilon)/2$. By \lemref{lem:sobolevest} we have
 \begin{equation}\label{eq:sobolevestmu2}
 \int_{\mathbb{S}^{d-1}}\|\mu_\theta\|^2_{H^\sigma(\R^{d-1})}\, d\theta\lesssim_{\varepsilon} 1,
 \end{equation}
 and for every $Q\in\DD_\Delta$
 \begin{equation}\label{eq:sobolevestmuQ2}
 \int_{\mathbb{S}^{d-1}}\|\mu_{Q,\theta}\|^2_{H^\sigma(\R^{d-1})}\, d\theta\lesssim_{\varepsilon} \mu(Q).
 \end{equation}
 \subsection{Light part} 
 
 We say that a cube $Q\in\DD_\delta$ is light, denoted by $Q\in\DD_{\delta,L}$, if
 \begin{equation*}
 \mu(Q)\le\ell(Q)^{d+\varepsilon}=\delta^{d+\varepsilon}.
 \end{equation*}
 Note that, by \eqref{eq:Frostlower}, this implies that
 \begin{equation}\label{eq:contentsmall}
 \HH^{d-\tau}_\infty(E\cap Q)\le\mu(\overline{Q})=\mu(Q)\le \delta^{d+\varepsilon},
 \end{equation}
 where the equality in the middle holds because by \eqref{eq:Frostupper} $\mu$ does not charge $\partial Q$, which is $(d-1)$-dimensional.
 We set
 \begin{equation*}
 E_0 = E\setminus \bigcup_{Q\in\DD_{\delta,L}}Q,
 \end{equation*}
 and 
 \begin{equation*}
 \cQ_\delta \coloneqq \DD_{\delta}\setminus\DD_{\delta,L},
 \end{equation*}
 so that $\cQ_\delta$ is the collection of all dyadic $\delta$-cubes intersecting $E_0$. Given $Q\in\DD_{\Delta}$, we set $\cQ_\delta(Q)=\{P\in\cQ_\delta\ :\ P\subset Q \}$.
 
 Consider families of tubes $\T_{\delta,\theta}$ as in Subsection \ref{sec:params}. We will say that a tube $T\in\T_{\delta,\theta}$ is {light} if
 \begin{equation*}
 \#\{P\in\cQ_\delta\ :\ P\cap T\neq\varnothing \}\le \delta^{-1+\tau+\varepsilon}.
 \end{equation*}
 We denote the family of light tubes by $\T_{L,\theta}$. When the direction $\theta$ is clear from context, we will just write $\T_L$.
 
 We set
 \begin{equation*}
 E_{L,\theta} = \bigg(E_0\cap\bigcup_{T\in\T_{L,\theta}}T\bigg)\cup\bigcup_{Q\in\DD_{\delta,L}}Q\cap E.
 \end{equation*}
 \begin{lemma}\label{lem:lightpart2}
 	We have 
 	\begin{equation*}
 	\HH^{d-\tau}_\infty(E_{L,\theta})\lesssim \delta^{\varepsilon}.
 	\end{equation*}
 \end{lemma}
 \begin{proof}
 	We estimate the content of light cubes and light tubes separately.
 	Regarding the light cubes, we have
 	\begin{equation*}
 	\HH^{d-\tau}_\infty\big(\bigcup_{Q\in\DD_{\delta,L}}Q\cap E\big)\le\sum_{Q\in\DD_{\delta,L}} \HH^{d-\tau}_\infty(Q\cap E)\overset{\eqref{eq:contentsmall}}{\le}\delta^{-d}\cdot \delta^{d+\varepsilon}=\delta^{\varepsilon}.
 	\end{equation*}
 	
 	Concerning tubes, observe that
 	\begin{equation*}
 	N\big(E_0\cap\bigcup_{T\in\T_L}T,\delta\big) \le \sum_{T\in\mathcal{T}_L}N(T\cap E_0,\delta)\lesssim \delta^{1-d}\cdot \delta^{-1+\tau+\varepsilon}=\delta^{-d+\tau+\varepsilon},
 	\end{equation*}
 	which gives
 	\begin{equation*}
 	\HH_\infty^{d-\tau}\big(E_0\cap\bigcup_{T\in\T_L}T\big)\le N\big(E\cap\bigcup_{T\in\T_L}T,\delta\big)\cdot \delta^{d-\tau}\lesssim \delta^{\varepsilon}.
 	\end{equation*}
 \end{proof}
 \subsection{Bad part}
 Fix a direction $\theta\in\mathbb{S}^{d-1}$. Let $Q\in\DD_\Delta$. If a tube $T\in\T_{\delta}$ satisfies
 \begin{equation}\label{eq:defsubstinter2}
 \#\{P\in\cQ_\delta(Q)\ :\ P\cap T\neq\varnothing \}\ge \delta^{\tau+\kappa-1+ 2\varepsilon},
 \end{equation}
 we will write $T\in\T(Q)$.
 \begin{lemma}\label{lem:normaltubes2}
 	Let $T\in\T_{\delta}\setminus\T_L$. Then, there exists $Q\in\DD_{\Delta}$ such that $T\in\T(Q)$.
 \end{lemma}
 \begin{proof}
 	Observe that, trivially, 
 	\begin{equation*}
 	\#\{Q\in\DD_\Delta\ :\ Q\cap T\neq\varnothing \}\lesssim \Delta^{-1}\sim \delta^{-\kappa}.
 	\end{equation*}
 	At the same time, since $T\notin\T_{L}$, we have $\#\{P\in\cQ_\delta\ :\ P\cap T\neq\varnothing \}> \delta^{-1+\tau+\varepsilon}.$ By pigeonholing, we get that for some $Q\in\DD_{\Delta}$
 	\begin{equation*}
 	\#\{P\in\cQ_\delta(Q)\ :\ P\cap T\neq\varnothing \}\gtrsim \frac{\#\{P\in\cQ_\delta\ :\ P\cap T\neq\varnothing \}}{\#\{Q'\in\DD_\Delta\ :\ Q'\cap T\neq\varnothing \}}\gtrsim \delta^{\kappa+\tau-1+\varepsilon}.
 	\end{equation*}
 \end{proof}
 Let $\mathcal{L}$ denote the set of affine lines parallel to $\theta$. We will write $\ell\in\cL_{Q,B}$ if the $\delta$-tube $T\in\T_{\delta}$ with $\ell\subset T$ satisfies $T\in\T(Q)$, and at the same time
 %	\begin{equation}\label{eq:defbadline}
 %		N(\ell(2\delta)\cap E\cap 3Q, \delta)\ge \delta^{(\kappa-1)(s-d+1) + \tau+ 2\varepsilon}\quad\text{and}\quad  \ell\cap E\cap \overline{3Q}=\varnothing.
 %	\end{equation}
 %	\begin{equation}\label{eq:defbadline}
 %	\#\{P\in\DD_\delta(Q)\ :\ P\cap\ell(2\delta)\neq\varnothing \}\ge \delta^{(\kappa-1)(s-d+1) + \tau+ 2\varepsilon}\quad\text{and}\quad  \ell\cap E\cap \overline{3Q}=\varnothing.
 %	\end{equation}
 \begin{equation}\label{eq:defbadline2}
 \ell\cap E\cap \overline{Q}=\varnothing.
 \end{equation}
 We set	
 \begin{equation*}
 L_{Q,B} = \bigcup_{\ell\in\cL_{Q,B}} \ell,\quad L_B=\bigcup_{Q\in\DD_{\Delta}}L_{Q,B}.
 \end{equation*}
 We define the bad part of $E$ as
 \begin{equation*}
 E_B \coloneqq E\cap L_B.
 \end{equation*}
 %	We claim that 
 %	\begin{equation*}
 %		\int_{\mathbb{S}^{d-1}} \HH_\infty^{d-\tau}(E_{B,\theta})\, d\theta\lesssim \delta^{\varepsilon}.
 %	\end{equation*}
 \begin{lemma}\label{lem:1}
 	For every $Q\in\DD_{\Delta}$ we have
 	\begin{equation*}
 	\HH_\infty^{d-1-\tau}(\pi_\theta(L_{Q,B}))\lesssim \delta^{\varepsilon}\|\mu_{Q,\theta}\|_{H^\sigma(\R^{d-1})}^2.
 	\end{equation*}
 \end{lemma}
 \begin{proof}
 	Arguing exactly the same as in the beginning of the proof of \lemref{lem:contentprojection}, we arrive at the inequality
 	\begin{multline}
 	I_1=\int \mu_{Q,\theta}\ast\varphi_{C\delta}\ast\varphi_{\eta}\ d\nu = I_2\\
 	\lesssim \delta^{a}\left(\int_{\R^{d-1}} |\widehat{\mu_{Q,\theta}}(\xi)|^{2}|\xi|^{1-\tau-\varepsilon}\ d\xi  \right)^{1/2}\left(\int_{\R^{d-1}} |\widehat{\nu}(\xi)|^{2}|\xi|^{(2a - 1 + \tau +(d-1)+\varepsilon) - (d-1)}\ d\xi  \right)^{1/2}.
 	\end{multline}
 	Recalling that $\sigma=(1-\tau-\varepsilon)/2$, and noting that for $a=1/2-\tau-\varepsilon$ we have
 	\begin{equation*}
 	\int_{\R^{d-1}} |\widehat{\nu}(\xi)|^{2}|\xi|^{(2a - 1 + \tau +(d-1)+\varepsilon) - (d-1)}\ d\xi\lesssim_{\varepsilon} \HH_\infty^{d-1-\tau}(\pi_\theta(L_{Q,B}))^{-1},
 	\end{equation*}
 	we get that
 	\begin{equation*}
 	I_2\lesssim_{\varepsilon}\delta^{1/2-\tau-\varepsilon}\HH_\infty^{d-1-\tau}(\pi_\theta(L_{Q,B}))^{-1/2}\|\mu_{Q,\theta}\|_{H^\sigma(\R^{d-1})}.
 	\end{equation*}
 	On the other hand, it follows from the definition of bad lines $\cL_{Q,B}$ that for each $\ell\in\cL_{Q,B}$
 	\begin{equation*}
 	\mu_Q(\ell(C\delta))\gtrsim \delta^{\tau+\kappa-1+2\varepsilon}\cdot \delta^{d+\varepsilon},
 	\end{equation*}
 	where we also used that $\mu_Q(P)\ge\delta^{d+\varepsilon}$ for all $P\in\cQ_\delta(Q)$. Hence, for each $x\in \pi_\theta(L_{Q,B})$
 	\begin{equation*}
 	\mu_{Q,\theta}\ast\varphi_{C\delta}\ast\varphi_{\eta}(x)\gtrsim \delta^{\tau+\kappa+3\varepsilon}.
 	\end{equation*}
 	Since $\nu$ is a probability measure, we get $I_1\gtrsim \delta^{\tau+\kappa+3\varepsilon}$. Comparing this with the estimate for $I_2$, we arrive at
 	\begin{equation*}
 	\HH_\infty^{d-1-\tau}(\pi_\theta(L_{Q,B}))\lesssim_{\varepsilon} \delta^{1-2\tau-2\varepsilon}\cdot\delta^{-2\tau-2\kappa-6\varepsilon}\|\mu_{Q,\theta}\|_{H^\sigma(\R^{d-1})}^2.
 	\end{equation*}
 	Recalling that $\kappa=1/6$ and $\tau = 1/6-5\varepsilon$, we get
 	\begin{equation*}
 	\HH_\infty^{d-1-\tau}(\pi_\theta(L_{Q,B}))\lesssim_{\varepsilon} \delta^{12\varepsilon}\|\mu_{Q,\theta}\|_{H^\sigma(\R^{d-1})}^2.
 	\end{equation*}
 \end{proof}
 \begin{lemma}\label{lem:estbad}
 	We have 
 	\begin{equation}\label{eq:badestimate3}
 	\int_{\mathbb{S}^{d-1}}\HH^{d-\tau}_\infty(E_{B,\theta})\ d\theta\lesssim_\varepsilon\delta^{\varepsilon}.
 	\end{equation}
 \end{lemma}
 \begin{proof}
 	First, observe that
 	\begin{equation*}
 	\HH^{d-\tau}_\infty(E_{B,\theta})\lesssim \HH^{d-1-\tau}_\infty(\pi_\theta(E_{B,\theta})).
 	\end{equation*}
 	Indeed, if $\cB$ is a nearly optimal covering of $\pi_\theta(E_{B,\theta})$ by cubes, then we can cover $E_{B,\theta}$ by a family of tubes $\mathbf{T}=\{T_Q : T_Q\coloneqq \pi_\theta^{-1}(Q)\cap [0,1]^d, Q\in\cB\}$ (recall that $E\subset [0,1)^d).$ At the same time, each $T_Q$ can be covered by $\lesssim \ell(Q)^{-1}$ many cubes of sidelength $\ell(Q)$. Hence,
 	\begin{equation*}
 	\HH^{d-\tau}_\infty(E_{B,\theta})\lesssim \sum_{Q\in\cB} \ell(Q)^{-1}\cdot \ell(Q)^{d-\tau}\lesssim \HH^{d-1-\tau}_\infty(\pi_\theta(E_{B,\theta})).
 	\end{equation*}
 	
 	Since $\pi_\theta(E_{B,\theta}) \subset \bigcup_{Q\in\DD_{\Delta}}\pi_\theta(L_{Q,B}) $, we have
 	\begin{equation*}
 	\HH^{d-1-\tau}_\infty(\pi_\theta(E_{B,\theta}))\le  \sum_{Q\in\DD_{\Delta}}\HH_\infty^{d-1-\tau}(\pi_\theta(L_{Q,B})),
 	\end{equation*}
 	and so by \lemref{lem:1}
 	\begin{multline*}
 	\int_{\mathbb{S}^{d-1}}\HH^{d-\tau}_\infty(E_{B,\theta})\ d\theta \lesssim \int_{\mathbb{S}^{d-1}}\HH^{d-1-\tau}_\infty(\pi_\theta(E_{B,\theta}))\ d\theta\\
 	\lesssim_\varepsilon \delta^{\varepsilon}\sum_{Q\in\DD_{\Delta}}\int_{\mathbb{S}^{d-1}} \|\mu_{Q,\theta}\|_{H^\sigma(\R^{d-1})}^2\, d\theta \overset{\eqref{eq:sobolevestmuQ2}}{\lesssim_{\varepsilon}} \delta^{\varepsilon}\sum_{Q\in\DD_{\Delta}}\mu(Q)\sim\delta^{\varepsilon}.
 	\end{multline*}
 \end{proof}
 \subsection{Good part}
 Fix $\theta\in\mathbb{S}^{d-1}$. We define the good part of $E$ as
 \begin{equation*}
 E_G \coloneqq E\setminus (E_B\cup E_L).
 \end{equation*}
 Recall that  $\T_{\delta}$ is a family of $\delta$-tubes parallel to $\theta$.
 \begin{lemma}\label{lem:cov2}
 	For any $T\in\T_\delta$ we have
 	\begin{equation*}
 	N(\vis_\theta(E)\cap E_G\cap T,\delta)\lesssim \delta^{\tau-1+2\varepsilon} + \delta^{\kappa-1}.
 	\end{equation*}
 \end{lemma}
 \begin{proof}
 	Assume that $T\cap E_G\neq\varnothing$, otherwise there is nothing to prove. Since $E_G\cap E_L=\varnothing$, we get that $T\notin\T_L$. Hence, by \lemref{lem:normaltubes2}, there exists $Q\in\DD_{\Delta}$ such that $T\in\T(Q)$, so that
 	\begin{equation}\label{eq:defsubstinter3}
 	\#\{P\in\cQ_\delta(Q)\ :\ P\cap T\neq\varnothing \}\ge \delta^{\tau+\kappa-1+ 2\varepsilon}.
 	\end{equation}
 	Let $Q_T\in\DD_{\Delta,T}$ be the ``$\theta$-highest'' of all cubes $Q\in\DD_\Delta$ satisfying \eqref{eq:defsubstinter3}, in the sense that it maximizes $\inf\{x\cdot\theta : x\in Q\}$ among all such cubes. Consider families $\DD^<_{\Delta,T},\DD^=_{\Delta,T},\DD^>_{\Delta,T}$ defined as in \eqref{eq:13}, except with $3Q, 3Q_T$ replaced by $Q, Q_T$.
 	
 	For cubes $Q\in\DD_{\Delta,T}^<$, we argue in the same way as in the proof of \lemref{lem:coveringgood} that $N(\vis_\theta(E)\cap E_G\cap T\cap Q,\delta)=0$.
 	
 	For cubes $Q\in\DD_{\Delta,T}^=$ we use the trivial bound
 	\begin{equation*}
 	N(\vis_\theta(E)\cap E_G\cap T\cap Q,\delta)\le N(T\cap Q,\delta)\lesssim \delta^{\kappa-1}
 	\end{equation*}
 	
 	Finally, for cubes $Q\in\DD_{\Delta,T}^>$ the inequality \eqref{eq:defsubstinter3} fails (this follows from the definition of $Q_T$), so that
 	\begin{equation*}
 	N(\vis_\theta(E)\cap E_G\cap T\cap Q,\delta)\le N(E_0\cap T\cap Q,\delta)\lesssim\delta^{\tau+\kappa-1+ 2\varepsilon},
 	\end{equation*}
 	where we used the fact that $E_G\subset E_0$, since $E_G\cap E_L=\varnothing$.  
 	
 	Putting the three estimates together, and noting that $\#\DD_{\Delta,T}^=\lesssim 1$ and $\#\DD_{\Delta,T}^<\lesssim \Delta^{-1}\sim \delta^{-\kappa}$, we get
 	\begin{equation*}
 	N(\vis_\theta(E)\cap E_G\cap T,\delta)\lesssim \delta^{\kappa-1}+ \delta^{\tau-1+2\varepsilon}.
 	\end{equation*}
 \end{proof}
 \begin{lemma}\label{lem:estgood}
 	We have
 	\begin{equation}\label{eq:44}
 	\HH_\infty^{d-\tau}(\vis_\theta(E)\cap E_G)\lesssim\delta^\varepsilon.
 	\end{equation}
 \end{lemma}
 \begin{proof}
 	Using \lemref{lem:cov2} we get
 	\begin{multline*}
 	\HH_\infty^{d-\tau}(\vis_\theta(E)\cap E_G)\le \sum_{T\in\mathcal{T}_\delta} N(\vis_\theta(E)\cap E_G\cap T,\delta)\cdot\delta^{d-\tau} \\
 	\lesssim \delta^{-d+1}\cdot ( \delta^{\tau-1+2\varepsilon} + \delta^{\kappa-1})\cdot\delta^{d-\tau}
 	= \delta^{2\varepsilon}  + \delta^{ \kappa-\tau}.
 	\end{multline*}
 	Recalling that $\tau = 1/6-5\varepsilon = \kappa - 5\varepsilon$, we get \eqref{eq:44}.
 \end{proof}
 
 Putting together estimates from \lemref{lem:lightpart2}, \lemref{lem:estbad}, and \lemref{lem:estgood}, we get
 \begin{equation*}
 \int_{\mathbb{S}^{d-1}} \HH_\infty^{d-\tau}(\vis_\theta(E))\, d\theta \lesssim_\varepsilon \delta^{\varepsilon}.
 \end{equation*}
 Taking $\delta\to 0$, and then $\varepsilon\to 0$, \thmref{thm:main2} follows.

%%% AUTHOR: optional appendix here
%\appendix %% you may comment this out if no Appendix
%\section*{Appendix}
%\section{Improving the constants}
%Material is placed here as needed.

%%% AUTHOR: optional acknowledgments here
\section*{Acknowledgments} %%  you may comment this out if no Ackno
I am grateful to Tuomas Orponen for many helpful discussions, and to Esa Järvenpää for reading a draft of this paper and giving useful comments.

%%% AUTHOR:
%%% Bibliography goes here. Note that the arXiv cannot process bibtex
%%% or biber bibliographies.  Example of acceptable bibliograpy format:
%\bibliographystyle{amsplain}

\newcommand{\etalchar}[1]{$^{#1}$}

%% AUTHOR: You can generate such a bibliography from a .bib file by 
%% running pdflatex/bibtex/pdflatex/pdflatex and then pasting the .bbl file
%% between \begin{thebibliography} and \end{bibliography}

%%% AUTHOR: Include a short description of each author following the
%%% structure below. Use the same short tags used previously.  
%%% Use \imageat{} and \imagedot{} instead of "@" and "." in
%%% email addresses-this replaces the symbols with graphics to avoid 
%%% e-mail address harvesting from the .pdf file
\begin{dajauthors}
\begin{authorinfo}[damian]
  Damian D\k{a}browski\\
  Department of Mathematics and Statistics\\ 
  University of Jyv\"askyl\"a\\
  Jyv\"askyl\"a, Finland\\
  damian\imagedot{}m\imagedot{}dabrowski\imageat{}jyu\imagedot{}fi \\
  \url{https://www.damiandabrowski.eu/}
\end{authorinfo}
\end{dajauthors}

\end{document}